\documentclass[11pt]{amsart}
\usepackage{dynkin-diagrams}
\usepackage[utf8]{inputenc}
\usepackage{comment}
\usepackage[style=alphabetic,firstinits,backend=bibtex, uniquename=init,
url=true,hyperref]{biblatex}
\DeclareFieldFormat{labelalpha}{\thefield{entrykey}}
\DeclareFieldFormat{extraalpha}{}
\DeclareFieldFormat{postnote}{#1}
\DeclareFieldFormat{multipostnote}{#1}
\usepackage{tikz}
\usepackage{tikz-cd}
\usepackage{cancel}
\usetikzlibrary{decorations.text,calc,arrows.meta}
\usepackage{stmaryrd}
\usepackage{array}
\usepackage{nicematrix}
\usepackage{varwidth}
\usepackage{filecontents}
\usepackage[titletoc]{appendix}
\usepackage{mathtools}
\usepackage{amsthm}
\usepackage{amsmath}
\usepackage{mathabx}
\usepackage{amsfonts}
\usepackage{amssymb}
\usepackage{mathrsfs}
\usepackage{multirow}
\usepackage{xypic}
\usepackage{color}
\usepackage{xcolor}
\usepackage{titlesec}
\usepackage{titletoc}
\usepackage[colorlinks]{hyperref}
\definecolor{cclr}{rgb}{25,25,112}
\hypersetup{
citecolor=[rgb]{0.15, 0.15, 0.68},
linkcolor=gray,   
urlcolor=magenta}

\bibliography{dim}
\newtheorem{lem}{Lemma}
\newtheorem{cor}{Corollary}
\newtheorem{prop}{Proposition}

\newtheorem{thm}{Theorem}
\newtheorem{exer}{Exercise}
\newtheorem{eg}{Example}

\newtheorem{dfn}{Definition}
\newtheorem{rmk}{Remark}

\def\presuper#1#2%
  {\mathop{}%
   \mathopen{\vphantom{#2}}^{#1}%
   \kern-\scriptspace%
   #2}

\newcommand{\Chtop}{{\operatorname{Ch}_{\operatorname{top}}}}

\newcommand{\Remain}{\operatorname{res}}

\newcommand{\RSw}{\operatorname{rsw}}
\newcommand{\CSw}{\operatorname{csw}}
\newcommand{\Tr}{\operatorname{tr}}

\newcommand{\Ell}{\operatorname{ell}}

\newcommand{\gr}{\operatorname{gr}}

\newcommand{\wt}[1]{\widetilde{#1}}

\newcommand{\Lie}{\operatorname{Lie}}

\newcommand{\Span}{\operatorname{span}}
\newcommand{\GSp}{\operatorname{GSp}}
\newcommand{\GO}{\operatorname{GSO}}
\newcommand{\Or}{\operatorname{O}}

\newcommand{\ad}{\operatorname{ad}}

\newcommand{\Gal}{\operatorname{Gal}}
\newcommand{\Gm}{\operatorname{G}_m}

\newcommand{\red}{{\operatorname{red}}}

\newcommand{\Aut}{\operatorname{Aut}}

\newcommand{\Res}{\operatorname{Res}}

\newcommand{\Dyn}{{\operatorname{Dyn}}}

\newcommand{\GSO}{{\operatorname{GSO}}}
\newcommand{\Ext}{{\operatorname{Ext}}}

\newcommand{\GR}{\operatorname{Gr}}

\newcommand{\Sp}{\operatorname{Sp}}
\newcommand{\Mat}{\operatorname{Mat}}
\newcommand{\GL}{\operatorname{GL}}

\newcommand{\SO}{\operatorname{SO}}

\newcommand{\Img}{\operatorname{Im}}

\newcommand{\cyc}{{\operatorname{cyc}}}
\newcommand{\bS}{\breve{S}}
\newcommand{\bG}{\breve{G}}
\newcommand{\bT}{\breve{T}}
\newcommand{\brB}{\breve{B}}
\newcommand{\bP}{\breve{P}}
\newcommand{\bM}{\breve{M}}

\newcommand{\Hom}{\operatorname{Hom}}

\newcommand{\spec}{\operatorname{Spec}}

\titleformat{\section}[runin]{\normalfont\bfseries}{\thesection.}{3pt}{}
\titleformat{\subsection}[runin]{\normalfont\bfseries}{\thesubsection.}{3pt}{}
\titleformat{\subsubsection}[runin]{\normalfont\bfseries}{\thesubsubsection.}{3pt}{}
\titleformat{\paragraph}[runin]{\normalfont\bfseries}{}{3pt}{}
\renewcommand{\thesection}{\arabic{section}}

\titleformat{\section}{\normalfont\large\bfseries}{\thesection.~~}{1em}{}
\newcommand{\cX}{\mathcal{X}}
\newcommand{\cY}{\mathcal{Y}}

\newcommand{\bFp}{\bar{\mathbb{F}}_p}

\newcommand{\bQ}{\mathbb{Q}}
\newcommand{\bZ}{\mathbb{Z}}

\newcommand{\lsup}[2]{{^{#1}\!#2}}

\dottedcontents{section}[2.5em]{}{2.9em}{1pc}

\begin{document}

\title{The topological Breuil-M\'ezard conjecture
for classical groups}
\author{Zhongyipan Lin}

\begin{abstract}
For unitary, orthogonal
and symplectic groups,
we compute the dimension
of the reduced Emerton-Gee stacks,
and give an explicit description of
their top-dimensional Chow group.
Our results are unconditional when $p\neq 2$.

The main innovation is a new method
for analyzing the vanishing locus of
Galois cup products
valued in high dimensions.
\end{abstract}

\maketitle

\tableofcontents

Let $G$ be a connected reductive group over a $p$-adic
field $F$, $p\neq 2$.
In this paper, we prove that
\begin{thm}[Corollary \ref{cor:main}]
If $\lsup LG$ is one of
\begin{itemize}
\item the $L$-dual of unitary groups $\lsup LU_n$,
\item (the neutral component of) the
orthogonal similitude groups $\GSO_n$, or
\item the symplectic similitude groups $\GSp_{n}$,
\end{itemize}
then
$$
\dim \cX_{\lsup LG} = [F:\bQ_p]\dim \bG/\brB
$$
where $\lsup LG=\bG \rtimes \Gal(K/F)$
is the L-group of $G$, $\brB$ is a Borel of $\bG$,
and $\cX_{\lsup LG}$
is the {\it reduced} Emerton-Gee stack
over $\bFp$.
Here $K/F$ is the quadratic splitting field of $G$
if $G$ is unitary, and $K=F$ if otherwise.
\end{thm}

\begin{thm}[Theorem \ref{thm:main2}]
Label the Dynkin diagram of $\bG$
by
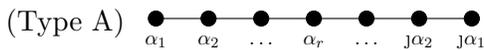
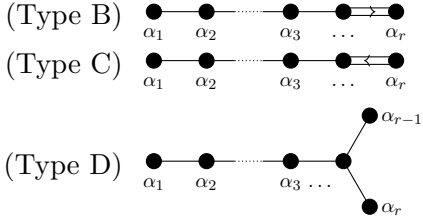
\begin{itemize}
\item[(Type A)] 
\dynkin[
  labels={\alpha_1,\alpha_2,\dots,
        \alpha_r, \dots
        , \j\alpha_2,\j\alpha_1},
  scale=2
]{A}{7}
\item[(Type B)] 
\dynkin[
  labels={\alpha_1,\alpha_2,\alpha_3,\dots,\alpha_r},
  scale=2
]{B}{}
\item[(Type C)] 
\dynkin[
  labels={\alpha_1,\alpha_2,\alpha_3,\dots,\alpha_r},
  scale=2
]{C}{}
\item[(Type D)] 
\dynkin[
labels={\alpha_1,\alpha_2,\alpha_3,\dots,\alpha_{r-1}, \alpha_{r}},
  scale=2
]{D}{}
\end{itemize}
There exist
natural isomorphisms
\begin{itemize}
\item[($A_{2n}$)]
$
(\alpha_1^\vee, \dots,\alpha_n^\vee, \Remain
):
\Chtop(\cX_{\lsup LU_{2n+1}})
\xrightarrow{\cong} X_1(K, G_m)^{\times n}
\times 
\Chtop(\cX_{\lsup LU_{1}})
$,
\item[($A_{2n+1}$)]
$
(\alpha_1^\vee, \dots,\alpha_n^\vee, \Remain
):
\Chtop(\cX_{\lsup LU_{2n+2}})
\xrightarrow{\cong} X_1(K, G_m)^{\times n}
\times 
\Chtop(\cX_{\lsup LU_{2}})
$,
\item[($B_{n}$)]
$
(\alpha_1^\vee, \dots,\alpha_{n-1}^\vee, \Remain
):
\Chtop(\cX_{\GSp_{2n}})
\xrightarrow{\cong} X_1(K, G_m)^{\times (n-1)}
\times 
\Chtop(\cX_{\GL_2})
$,
\item[($C_{n}$)]
$
(\alpha_1^\vee, \dots,\alpha_{n-1}^\vee, \Remain
):
\Chtop(\cX_{\GSO_{2n+1}})
\xrightarrow{\cong}X_1(K, G_m)^{\times (n-1)}
\times 
\Chtop(\cX_{\GSO_3})
$,
\item[($D_{n}$)]
$
(\alpha_1^\vee, \dots,\alpha_{n-1}^\vee, \Remain
):
\Chtop(\cX_{\GSO_{2n}})
\xrightarrow{\cong} X_1(K, G_m)^{\times (n-2)}
\times 
\Chtop(\cX_{\GSO_4})
$.
\end{itemize}
See Definition
\ref{dfn:x1} for definitions.
\end{thm}

\begin{cor}
    [Corollary \ref{cor:main3}]
All top-dimensional irreducible components
of $\cX_{\lsup LG}$
are covered by algebraic cycles
defined by reduction mod $p$
of de Rham stacks of regular Hodge type.
\end{cor}

Our intermediate results are strong enough
to establish the unconditional existence of crystalline
lifts for classical groups,
which has its own unique difficulties
and is worked out in a separate paper.
The majority of this paper
attacks all three types of groups simultaneously,
with the group-specific portions
concentrated in the first few section.

The strategy used by \cite{EG23}
for $\bG=\GL_n$ is induction along
its  maximal proper parabolics
$\bP=U\rtimes \bM$,
exploiting the abelian-ness of 
the unipotent radical $U$;
the coarse moduli space
of all fibers of
$\cX_{U\rtimes \bM}\to\cX_{\bM}$
are vector spaces
and
the induction process is largely analyzing
the rank of these vector spaces.

The major challenge for extending this strategy to more
general groups is that $U$ is rarely abelian.
For classical groups other than $\GL_n$, $U$ is typically
nilpotent of class $2$.
Now 
the coarse moduli space
of fibers of
$\cX_{U\rtimes \lsup LM}\to\cX_{\lsup LM}$
are affine cones;
here $\lsup LM = \bM\rtimes \Gal(K/F)$.
Write $Z:=[U,U]$ and $V:=U/Z$.
The coarse moduli space of fibers of both
\begin{align*}
g&: \cX_{V\rtimes \lsup LM}\to\cX_{\lsup LM}\text{, and}
\\
h&: \cX_{U\rtimes \lsup LM}\to\cX_{V\rtimes\lsup LM}
\end{align*}
are indeed vector spaces.
Write $\cX_{V\rtimes\lsup LM}^c$
for the image of $h$
(which is automatically a closed substack).
The new input needed is the fiberwise codimension
of
$\cX_{V\rtimes\lsup LM}^c$
in
$\cX_{V\rtimes\lsup LM}$
along $g$,
and for the proof to work, we need 
a really non-trivial estimate of this codimension.

This paper
is an expansion and vast generalization of
the last section of the initial draft
of \cite{L25A}
combined with the last section of the initial draft
of \cite{L25B}.
This paper is presented independently because 
its methods and results are independent to
other parts of these studies,
thereby avoiding unpleasant notational conflicts and verbosity.

\section{A source of inspiration}

Before we dive into the proof,
we revisit $\GL_5$
as a source of inspiration.
We fix a Galois character $\chi:\Gal_K\to \bFp^\times$,
and consider all Galois representations
of the form
\begin{align}
\label{eq:galois-rep}
\left(
\begin{array}{cc|c|cc}
\chi_\cyc & 0 & \multicolumn{1}{c|}{*_1} & * & * \\ 
0 & \chi_\cyc & \multicolumn{1}{c|}{*_2} & * & * \\ \hline
 &  & \multicolumn{1}{c|}{\chi} & *_3 & *_4 \\ \cline{3-5}
& \multicolumn{1}{c}{} & & 1 & 0 \\
& \multicolumn{1}{c}{} & & 0 & 1
\end{array}
\right)
\end{align}
where $\chi_\cyc:\Gal_K\to \bFp^\times$
is the cyclotomic character.
So, 
$\begin{bmatrix}*_1 \\ *_2\end{bmatrix}
\in \Ext^1(\chi, \chi_\cyc^{\oplus 2})
=H^1(\Gal_K, \chi^{-1}\chi_\cyc)^{\oplus 2}$
and 
$\begin{bmatrix}*_3 & *_4\end{bmatrix}
=H^1(\Gal_K, \chi)^{\oplus 2}$.
We have
$V= \chi^{\oplus 2}\oplus \chi^{-1}\chi_\cyc^{\otimes 2}$
and 
$Z = \Hom(1^{\oplus 2}, \chi_\cyc^{\oplus 2})$.

We regard $X:=H^1(\Gal_K, V)$
as an affine space,
and the collection
$$X^c:=\{(*_1, *_2, *_3, *_4) |*_1, *_2, *_3, *_4
\text{~can fit in \eqref{eq:galois-rep}}
\}$$
forms an affine cone inside $X$.
The question is $\dim X - \dim X^c = ?$.

First, consider the part of the cone where $*_3=*_4=0$.
In this case, the choice in $(*_1, *_2)$ is arbitrary:
so the dimension is 
$\dim H^1(\Gal_K, \chi^{-1}\chi_\cyc)^{\oplus 2}$
and the codimension is
$\dim H^1(\Gal_K, \chi)^{\oplus 2}$.

Second, consider the part of the cone where $*_3=0$ and $*_4\neq 0$;
\[
\left(
\begin{array}{cc|c|c}
\chi_\cyc & 0 & \multicolumn{1}{c|}{*_1} &  \\ 
0 & \chi_\cyc & \multicolumn{1}{c|}{*_2} & \\ \hline
 &  & \multicolumn{1}{c|}{\chi}  & *_4 \\ \cline{3-4}
& \multicolumn{1}{c}{} &  & 1
\end{array}
\right)
\]
$*_2$ is subject to a single constraint
due to local Tate duality pairing $*_2\cup*_4 = 0$.
It is less obvious but $*_1$
is also subject to a single,
transversal constraint.
So the codimension is
$2 + \dim H^1(\Gal_K, \chi)$.

It is not enough to just consider the cases where
$*_3=0$,
we need to consider all $1$-dimensional subspaces of 
the representation space
$\Mat_{1\times 2}$ of $(*_3, *_4)$,
which form the Grassmannian variety $\GR(1,2)$.
So the codimension of the combined loci
is equal to 
$2+\dim H^1(\Gal_K, \chi) - \dim \GR(1,2)
=1+\dim H^1(\Gal_K, \chi)$.

Third, consider the part of the cone where 
$(*_3, *_4)\neq (0,0)$;
indeed we impose the stronger full-rank
assumption that 
$\Span(*_3, *_4)$ has dimension $2$.
It is even less obvious, but
the constraints of $*_3$ and $*_4$
imposed 
on $*_1$ and $*_2$
are transversal.
So the codimension is $2\times2=4$.
The upshot is the codimension is
$
\min
\begin{cases}
2\dim H^1(\Gal_K, \chi) \\
1+\dim H^1(\Gal_K, \chi) \\
4.
\end{cases}
$

We can replace $\chi$ by
a higher rank Galois representation
$\rho: \Gal_K\to\GL_n(\bFp)$
and set 
$s:=\dim H^0(\Gal_K, \rho)$,
$t:=\dim H^2(\Gal_K, \rho)$.
The following generalization
will {\it not} be used in this paper,
and
we leave it an an exercise to the curious reader.

\begin{exer}
The locus of
$X^c:=\{(*_I, *_{II})\}\subset H^1(\Gal_K, V)=:X$ 
consisting of extension classes
that fit in
a Galois representation
\[
\begin{bmatrix}
\chi_\cyc^{\oplus m} & *_I & * \\
& \rho & *_{II} \\
& & 1^{\oplus m}
\end{bmatrix}
\]
has codimension at least
$$
\min_{0\le k\le m} (m-k)^2 + k(s+t+[K:\bQ_p]m).
$$
\end{exer}

\section{$\Delta$-involution structure for classical groups}
\label{sec:classic}
We first analyse classical groups one by one,
and extract axiomatically common structures.
We use the convenient convention that
$\GSp_0 = \GO_0=\GL_1$, and $\Sp_0=\SO_0=U_0=1$.

Denote by $w_n=\begin{bmatrix}& & 1 \\& \cdots & \\1&&\end{bmatrix}$
the antidiagonal matrix whose non-zero entries
are $1$;
and
denote by 
$w_n'=\begin{bmatrix}& & 1 \\& -1 & \\\cdots&&\end{bmatrix}$
the antidiagonal matrix whose non-zero entries
are alternating $1$, $-1$, $1$, $-1$,$\cdots$.

Define $\wt P_m :=
\begin{bmatrix}
\GL_m & \Mat_{m\times (n-2m)} & \Mat_{m \times m} \\
& \GL_{n-2m} & \Mat_{(n-2m)\times m} \\
& & \GL_m
\end{bmatrix}\subset \GL_n$.
Denote by $\wt U_m$
the unipotent radical of $\wt P_m$; we have
\begin{eqnarray*}
\wt Z_m :=\gr^0 \Lie \wt U_m &=& \Mat_{m\times m} \\
\wt V_m :=\gr^1 \Lie \wt U_m &=& 
\Mat_{m\times (n-2m)} \times \Mat_{(n-2m)\times m}=:
\wt V_I\times\wt V_{II}.
\end{eqnarray*}
Here we set $\wt V_I = \Mat_{m\times (n-2m)}$
and $\wt V_{II} = \Mat_{(n-2m)\times m}$.

\begin{eg}[Symplectic groups]

Let $n$ be an even integer.
The Dynkin diagram for $\GSp_{n}$
is 
\dynkin[
  labels={\alpha_1,\alpha_2,\alpha_3,\dots,\alpha_{n/2}},
  scale=2
]{B}{}.
Thus the maximal proper Levi subgroups of $\GSp_{n}$
are of the form
$$\GL_m \times \GSp_{n-2m} =: \bM_m,$$
for $m \le n/2$.
Write $\bP_m=U_m\rtimes \bM_m$ for the corresponding parabolic subgroup.
We have 
$
\bP_m =
\GSp_{n} \cap \wt P_m
$
if the following presentation of $\GSp_{n}$ is used:
$$
\GSp_{n} = \{X\in \GL_{n}|
X^t 
\begin{bmatrix}
& & w_m \\
& w'_{n-2m} & \\
-w_m & &
\end{bmatrix}
X = 
\lambda
\begin{bmatrix}
& & w_m \\
& w'_{n-2m} & \\
-w_m & &
\end{bmatrix}
$$

Define an $\Delta:=\{1, \j\}$-action on $\Lie \wt U_m$ by
\begin{eqnarray}
\label{eqn:symp}
    \j (x, y) :=& (y^tw'_{n-2m}, -w'_{n-2m} x^t),
& (x, y)\in \Mat_{m\times (n-2m)} \times \Mat_{(n-2m)\times m}\\
\j z :=& w_mz^tw_m,& z\in \Mat_{m\times m}.
\end{eqnarray}

We have
\begin{eqnarray*}
\wt Z_m^\Delta=(\gr^0 \Lie \wt U_m)^\Delta &=& \gr^0 \Lie U_m=:Z_m \\
\wt V_m^\Delta=(\gr^1 \Lie \wt U_m)^\Delta &=& \gr^1 \Lie U_m=:V_m \\
\j \wt V_I &=& \wt V_{II}.
\end{eqnarray*}
\end{eg}

\begin{eg}[Orthogonal groups]
The Dynkin diagram for $\GSO_n$
is 
\dynkin[
  labels={\alpha_1,\alpha_2,\alpha_3,\dots,\alpha_{n-1}/2},
  scale=2
]{C}{}
 or 
\dynkin[
labels={\alpha_1,\alpha_2,\alpha_3,\dots,\alpha_{n/2-1}, 
\alpha_{n/2}},
  scale=2
]{D}{}.
The maximal proper Levi subgroups of $\GSO_n$
are of the form
$$\GL_m \times \GSO_{n-2m}=:\bM_m.$$
Write $\bP_m=U_m\rtimes \bM_m$ for the corresponding parabolic
 subgroup.
We use the following presentation of $\GSO_{n}$:
$$
\GSO_{n} =\text{the neutral component of~} \{X\in \GL_{n}|
X^t 
w_n
X = \lambda 
w_n
\}
$$
(warning: for example $\{X\in \GL_{2}|
X^t 
w_2
X = \lambda 
w_2
\} = (\GL_1 \times \GL_1) \rtimes \bZ/2\bZ$
is not connected).

We have 
$
\bP_m
=\GSO_{n} \cap \wt P_m
$
where
$\wt P_m$ is the same as before.

Define an $\Delta:=\{1, \j\}$-action on $\Lie \wt U_m$ by
\begin{eqnarray}
\label{eqn:orth}
\j (x, y) :=& (-y^tw_{n-2m}, - w_{n-2m}x^t),& (x, y)\in \Mat_{m\times (n-2m)} 
\times \Mat_{(n-2m)\times m}\\
\j z :=& -w_mz^tw_m,& z\in \Mat_{m\times m}.
\end{eqnarray}

We have
\begin{eqnarray*}
\wt Z_m^\Delta=(\gr^0 \Lie \wt U_m)^\Delta &=& \gr^0 \Lie U_m=:Z_m \\
\wt V_m^\Delta=(\gr^1 \Lie \wt U_m)^\Delta &=& \gr^1 \Lie U_m=:V_m \\
\j \wt V_I &=& \wt V_{II}.
\end{eqnarray*}
\end{eg}

\begin{eg} [Unitary groups]
Now assume $U_n$ is a quasi-split unitary group over $F$
which splits over the quadratic extension $K/F$
(it is tamely ramified or unramified since $p\ne 2$).

The $L$-group $\lsup LU_n = \GL_n \rtimes \{1, \j\}$
where $\j$ acts on $\GL_n$ via $w_n(-)^{-t}w_n^{-1}$.
The Dynkin diagram of $U_n$ is a chain of $(n-1)$-vertices
(\dynkin{A}{}),
and $\Delta=\Gal(K/F)$ acts on it by reflection.
The maximal proper $\Delta$-stable subsets of $\Dyn(U_n)$
are given by removing either two symmetric vertices, or the middle vertex.
Therefore, the Levi subgroups of maximal proper $F$-parabolics
of $U_n$ are of the form
$$
M_m:=\Res_{K/F}\GL_m \times U_{n-2m}.
$$
Write $\lsup LP_m=\wt U_m\rtimes \lsup LM_m=
\wt P_m \rtimes \{1, \j\}$
for the corresponding parabolic.

We will also write $U_m=\wt U_m$
for the sake of uniformity
when dealing with all groups simultaneously;
despite $U$ means both unitary group
and unipotent radical,
it should be clear from the context.

In this paragraph, we
fix an $L$-parameter $\tau:\Gal_K\to \lsup LM_m(\bFp)$,
which induces the (adjoint) action of $\Gal_K$
on $\wt U_m(\bFp)$.
$\Delta$ 
acts on the Galois cohomology
$H^1(\Gal_K, \wt U_m(\bFp))$
via $\j[c] = (\gamma\mapsto \j \cdot c[\j^{-1}\gamma \j])$,
and by 
\cite[Theorem 3.15]{Ko02},
we have
$$
H^i(\Gal_F, \gr^j\Lie \wt U_m(\bFp))
 =
H^i(\Gal_K, \gr^j\Lie \wt U_m(\bFp))^{\Gal(K/F)}
=
H^i(\Gal_K, \gr^j\Lie \wt U_m(\bFp))^{\Delta},
$$
for all $i$ and $j$.
Moreover,
$$
\j H^1(\Gal_K, \wt V_I) = H^1(\Gal_K, \wt V_{II}).
$$
\end{eg}

\section{The affine cone problem I: axiomization}

Denote by $\lsup LG_n$ any of $\GSp_n$, $\GO_n$ or $\lsup LU_n$.
In this section, we don't notationally
distinguish an algebraic group
with its $\bFp$-points;
so, for example, $H^1(\Gal_F, U_m)$
is always understood as
$H^1(\Gal_F, U_m(\bFp))$.

We extract some common structure
from the case-by-case study in Section \ref{sec:classic}
of $\GSp_n$, $\GO_n$ and $\lsup LU_n$
(recall $K$ is the splitting field, and is either $F$ or a quadratic 
extension).
Set 
$$\lsup L{\wt P}_m := \wt P_m \rtimes \Gal(K/F)$$
in each of these cases for the sake of uniformity.

We have already seen the following features:
\begin{enumerate}
\item[(D1)] The maximal parabolics are of the form
$$
\lsup LP_m
= U_m\rtimes \lsup LM_m
= \lsup LG_n \cap \lsup L{\wt P}_m,
$$
\item[(D2)]
we have
$$
H^i(\Gal_K, \gr^j\Lie \wt U_m)^\Delta
=H^i(\Gal_F, \gr^j\Lie U_m)
$$
for all $i,j$,
\item[(D3)]
$\j H^1(\Gal_K, \wt V_I) = H^1(\Gal_K, \wt V_{II})$.
\end{enumerate}

We will fix an $L$-parameter
$$
\tau: \Gal_F \to \lsup LG_{n-2m}~ (=\lsup LG_{n-2m}(\bFp))
$$
as well as
a character $\alpha:\Gal_K \to \bFp^\times$.
Denote by $(\alpha, \tau)_m$
the parameter
$\Gal_F \to \lsup L(\Res_{K/F}\GL_m\times G_{n-2m})$
that has the matrix presentation
$
\begin{bmatrix}
\alpha^{\oplus m} & 0 & 0 \\
0 & \tau & 0 \\
0 & 0 &\beta^{\oplus m}
\end{bmatrix}
$
under the embedding
$\lsup LP_m \subset \lsup L{\wt P}_m$;
we remark that
here $\beta$ is completely determined by
$\alpha$ and $\tau$.
Let 
$$
X:=H^1(\Gal_F, V_m)
$$
be the vector space of all extensions
of $(\alpha, \tau)_m$
to $\Gal_F \to V_m\rtimes \lsup LM_m$,
and denote by
$X^c\subset X$
the affine cone consisting
of extensions that can be further extended
to $\Gal_F \to \lsup LP_m$.
It is clear that either
$H^2(\Gal_F, \wt Z_m)=0$
or
$\dim H^2(\Gal_F, \wt Z_m)=m^2$;
and in the first case, $X^c=X$.
So, we will assume we are in the second case,
and in particular
$\alpha = \beta (1) = \beta\cdot\chi_\cyc$.

\subsection{Eigenspace decomposition}
Since $p>2$,
we can decompose $H^1(\Gal_K, \wt V_m)$
according to the eigenvalues $\pm1$ of $\j$:
$$
H^1(\Gal_K, \wt V_m)
=
H^1(\Gal_K, \wt V_m)^+
\oplus
H^1(\Gal_K, \wt V_m)^-
:=
\{c|\j\cdot c = c\}
\oplus
\{c|\j\cdot c = -c\}.
$$
On the other hand,
we have
$
\wt V_m
=
\wt V_I
\oplus
V_{II}.
$
The upshot is we have
the following commutative diagram of isomorphisms
of vector spaces:
$$
\xymatrix{
H^1(\Gal_F, V_m) \ar@{=}[rr]
&& H^1(\Gal_K, \wt V_m)^\Delta \ar@{=}[d] \\
H^1(\Gal_K, \wt V_I)
\ar[d]^{\j}
 \ar@{->}[rr]_{\cong}^{\substack{\Tr:c\mapsto (c + \j\cdot c)/2 }}
&& H^1(\Gal_K, \wt V_m)^+ \ar@{-->}[d]^{\cong} \\
H^1(\Gal_K, \wt V_{II})
 \ar@{->}[rr]_{\cong}^{\substack{c\mapsto (-c + \j\cdot c)/2 }}
&& H^1(\Gal_K, \wt V_m)^-
}
$$
For each $c\in H^1(\Gal_K, \wt V_m)$,
write $c=c_I+c_{II}$
where $c_I\in H^1(\Gal_K, \wt V_I)$
and
$c_{II}\in H^1(\Gal_K, \wt V_{II})$.

\subsection{Cup product and obstructions}
By the long exact sequence for Galois cohomology,
the obstruction to extending
classes $c\in H^1(\Gal_F, V_m)$
to $H^1(\Gal_F, U_m)$
lies in $H^2(\Gal_F, Z_m)$:
$$
\xymatrix{
H^1(\Gal_F, U_m) \ar[r]
& H^1(\Gal_F, V_m) \ar[rr] \ar@{=}[d]
&& H^2(\Gal_F, Z_m) \ar@{=}[d] \\
H^1(\Gal_K, \wt U_m)^\Delta \ar[r]
& H^1(\Gal_K, \wt V_m)^\Delta \ar[rr]
\ar@{^{(}->}[d]
&& H^2(\Gal_K, \wt Z_m)^\Delta
\ar@{^{(}->}[d]
\\
H^1(\Gal_K, \wt U_m) \ar[r]
& H^1(\Gal_K, \wt V_m) \ar[rr]^{c_I+c_{II}\mapsto c_I\cup c_{II}}
&& H^2(\Gal_K, \wt Z_m)
}
$$
where $c_I\cup c_{II}$ coincides with the usual
cup product
$$
\cup:
H^1(\Gal_K, \wt V_I) \times
H^1(\Gal_K, \wt V_{II})
\to H^2(\Gal_K, \wt Z_m)
$$
induced from the bilinear map
$
\wt V_I \times \wt V_{II} \to \wt Z_m.
$

Since $H^2(\Gal_F, Z_m)$
is a vector subspace of $H^2(\Gal_K, \wt Z_m)$,
a parameter
$c=c_I+c_{II}\in H^1(\Gal_F, V_m)\subset H^1(\Gal_K, \wt V_m)$
can be extended to a parameter
in $H^1(\Gal_F, U_m)$
if and only if $c_I\cup c_{II}=0$.

The reader can verify by unravelling definitions
that
\begin{itemize}
\item[(D4)]
$\j(c_I\cup c_{II}) = \j c_{II} \cup \j c_I$.
\end{itemize}

Recall that $\wt V_{I} = \Mat_{m\times(n-2m)}
=
\left(
\begin{matrix}
\wt V_{I,1}
\\ 
\dots 
\\
\wt V_{I,m}
\end{matrix}
\right)
$
where $\wt V_{I,i}\cong \Mat_{1\times (n-2m)}$ is the $i$-th row,
and
$\wt V_{II} = \Mat_{(n-2m)\times m}
=
\left(
\begin{array}{c|c|c}
    \wt V_{II,1} & \dots & \wt V_{II,m}
\end{array}
\right)
$
where $\wt V_{II,j}$ is the $j$-th column.
We have 
$$
H^1(\Gal_K, \wt V_{I, i}) \cup 
H^1(\Gal_K, \wt V_{II, j}) = 
H^2(\Gal_K, \bFp e_{i,j})
$$
where $e_{i,j}\in \wt Z_m$ is the elementary matrix with non-zero entry
at position $(i,j)$.

Because of our choice of the matrix presentation of the groups,
we have
\begin{itemize}
\item[(D5)]
$\j \wt V_{I,i} = \wt V_{II, m+1-i}$,
\item[(D6)]
    $\j \bFp e_{i,j} = \bFp e_{m+1-j, m+1-i}$;
\end{itemize}
(we remark that if we change the matrix presentation,
we can arrange it so that
$\j \wt V_{I,i} = \wt V_{II, \sigma(i)}$
for an arbitary permutation $\sigma$
by replacing $w_n$ by another permutation matrix.)

Again because of our choice of presentation of the groups,
we have spinned down
isomorphisms
$\wt V_{I,i_1} \xrightarrow{\RSw_{i_1, i_2}} \wt V_{I,i_2}$
and
$\wt V_{II,j_1} \xrightarrow{\CSw_{i_1, i_2}} \wt V_{II,j_2}$
that simply swap the corresponding rows or columns
in the matrix presentation.

\subsection{Inner product / alternating form structure}
We have $\wt V_I \cong \wt V_{II}^\vee(1)$
as a $\Gal_K$-module,
and thus by local Tate duality,
for each $(i,j)$, the cup product
$$
\cup: H^1(\Gal_K,\wt V_{I,i}) \times 
H^1(\Gal_K,\wt V_{II, j}) \to
H^2(\Gal_K, \bFp e_{i,j})
$$
is a {\it non-degenerate} pairing.

We consider the following bilinear form
\begin{align*}
H^1(\Gal_K,\wt V_{I,1}) 
\times 
H^1(\Gal_K,\wt V_{I,1})
&\xrightarrow{}
H^2(\Gal_K, \bFp e_{1,n})
\cong \bFp e_{1,n}\\
(c_1, c_2)
    &\mapsto
(c_1\cdot c_2) e_{1,n}:=
c_1 \cup \j(c_2)
\end{align*}

\begin{lem}
\label{lem:inner}
If $G$ is either symplectic or unitary,
then $(c_1, c_2)\mapsto c_1\cdot c_2$
is a non-degenerate symmetric form;
and if $G$ is orthogonal,
then  $(c_1, c_2)\mapsto c_1\cdot c_2$
is a non-degenerate alternating form.
\end{lem}

\begin{proof}
The line $\bFp e_{1,n}$
is fixed by $\j$-involution by Axiom (D6);
Indeed direct computation shows
$\j e_{1,m}=e_{1,m}$ in the symplectic case
and
$\j e_{1,m}=-e_{1,m}$ in the orthogonal and unitary case.

Orthogonal case:
by Axiom (D4),
$-c_1\cup \j c_2=\j (c_1 \cup \j c_2) 
= \j^2 c_2 \cup \j c_1 = c_2\cup \j c_1$.

Symplectic case:
by Axiom (D4),
$c_1\cup \j c_2=\j (c_1 \cup \j c_2) 
= \j^2 c_2 \cup \j c_1 = c_2\cup \j c_1$.

The unitary case is completely similar,
and we leave it as an exercise.
\end{proof}

\section{The affine cone problem II: the inner product case}

Recall 
$X^c = \{\Tr(c)\in X=H^1(\Gal_F, V_m) | c \cup \j c = 0\}$.

Say $h = \dim H^1(\Gal_K, \wt V_{I,1})$;
so $\dim X = hm$.

We assume $G$ is either unitary or symplectic.
By the Gram-Schmidt process and Lemma \ref{lem:inner},
there exists an orthonormal basis
$u_1, \dots, u_h \subset H^1(\Gal_K, \wt V_{I,1})$
such that $u_i \cup \j u_j = \delta_{i,j} e_{1,n}$.
Denote by $u_i^s:=\RSw_{1, s}(u_i)\in H^1(\Gal_K, \wt V_{I,s})$
the translation of $u_i$ to the $s$-th row.

Write $c = c_1+\RSw_{1,2}c_2+\dots+\RSw_{1,m}c_m$
where $c_i \in H^1(\Gal_K, \wt V_{I,1})$ for $i=1,\dots,m$.
In other words, 
$c_i$ is the $i$-th row of $c$.

Consider the resolution
$H^1(\Gal_F, V_m)\times \GR(k, m) \to H^1(\Gal_F, V_m)$
and set
$$
\wt X_k\subset H^1(\Gal_F, V_m)\times \GR(k, m)
$$
to be the sub-variety of pairs $(\Tr c, S)$
such that 
\begin{itemize}
\item $S\subset \wt V_I$ is a $\Gal_K$-stable subspace
of dimension $k$,
\item $c\in H^1(\Gal_K, S) 
\subset H^1(\Gal_K, \wt V_I) \xrightarrow[\cong]{\Tr}
H^1(\Gal_F, V_m)$.
\end{itemize}
Set $\wt X^c_k$ to be the pullback
of $X^c$ to $\wt X_k$.

\subsection{Step 1: base case}
The equation in $\lambda$
$$
(c_1+\lambda u_1)\cup \j 
(c_1+\lambda u_1)
=
c_1\cup \j c_1 + 2\lambda c_1\cup \j u_1
+\lambda^2 e_{1,n}
=
(|c_1|^2 + 2 c_1\cdot u_1 \lambda + \lambda^2)e_{1,n}
$$
has exactly two solutions
counted with multiplicities.

\subsection{Step 2: good configuration}
We replace $c_1$ by $c_1+ \lambda u_1$
for some $\lambda$ to
assume $c_1\cup\j c_1=0$.
Next consider indeterminants $\lambda_1, \lambda_2$
and the equation
\begin{align*}
&(c_1+\RSw_{1,2}c_2+\lambda_1 u_1^2 +\lambda_2 u_2^2)\cup \j 
(c_1+\RSw_{1,2}c_2+\lambda_1 u_1^2 +\lambda_2 u_2^2)\\
=&
\CSw_{n,n-1}(c_1\cup\j (c_2+\lambda_1 u_1 + \lambda_2u_2)) +
\\
&\RSw_{1,2}\j(c_1\cup\j (c_2+\lambda_1 u_1 + \lambda_2u_2)) +
\\
&
\RSw_{1,2}\CSw_{n,n-1}
(c_2+ \lambda_1u_1 + \lambda_2u_2)\cup \j
(c_2+ \lambda_1u_1 + \lambda_2u_2)
\end{align*}
which is equivalent to
$$
\begin{cases}
c_1 \cdot c_2 + c_1\cdot u_1\lambda_1 
+ c_1\cdot u_2\lambda_2  = 0 \\
|c_2|^2 + 2c_2\cdot u_1 \lambda_1 + 2 c_2\cdot u_2 \lambda_2
+ \lambda_1^2+\lambda_2^2=0.
\end{cases}
$$
There are two scenarios:

\begin{itemize}
\item This system of equations
has infinitely many solutions,
which happens under the 
necessary condition that
$(c_1\cdot u_1)^2+(c_2\cdot u_2)^2=0$;
the reader can, for example,
use long division to get this condition; or
\item This system of equation has finitely many solutions.
\end{itemize}

Since we have already pinned down 
the value of $c_1\cdot u_1$
to guarantee $c_1\cdot c_1=0$;
the aforementioned necessary condition for the
infinite solution scenario
pinned down $c_1\cdot u_2$
with at most $2$ choices.
So, if we exchange the role of $c_1$ and $c_2$,
then this system of equation has only finitely many solutions.
Let's make this observation more rigorous:

\begin{dfn}
We say $(c_1,\dots, c_r; u_1,\dots,u_r)$
is a {good configuration}
if the equation
$$
L(c_1,\dots,c_r; u_1, \dots, u_r; \lambda_{1\le j\le i\le r})\cup\j
L(c_1,\dots,c_r; u_1, \dots, u_r; \lambda_{1\le j\le i\le r})=0
$$
where
\begin{align*}
L(c_1,\dots,c_r; u_1, \dots, u_r;
\lambda_{1\le j\le i\le r}) :=
&c_1 + \lambda_{11} u_1+ \\
&\RSw_{1,2}(c_2+\lambda_{12} u_1 + \lambda_{22} u_2)+\dots+\\
&\RSw_{1,r}(c_r + \lambda_{r1}u_1+\dots+\lambda_{rr}u_r)
\end{align*}
has at most finitely many solutions
in $(\lambda_{ij})_{1\le j\le i\le r}$.
\end{dfn}

Our discussion in this section can be summarized as
\begin{prop}
Either $(c_1, c_2; u_1, u_2)$ is a good configuration,
    or $(c_2, c_1, u_1, u_2)$ is a good configuration.
\end{prop}

\subsection{Step 3}
Now we are ready to make the general induction process.
The induction hypothesis is
$$
(c_1,\dots,c_{r-1}; u_1, \dots, u_{r-1})
$$
is a good configuration.
Define
$$
c'_i = c_i + \lambda_{i, 1} u_1 +\dots + \lambda_{i,i}u_i
~(1\le i< r)
$$
where $\{\lambda_{i,j}\}$ ($1\le j\le i < r$)
is one of the finitely many solutions
of the equation
$$
L(c_1,\dots,c_{r-1}; u_1, \dots, u_{r-1};
\lambda_{1\le j\le i< r})\cup\j
L(c_1,\dots,c_{r-1}; u_1, \dots, u_{r-1};
\lambda_{1\le j\le i< r})=0.
$$

Set $C_{ij}:= c'_i\cdot u_j$ for ease of notation.

The equation
$$
L(c_1,\dots,c_{r}; u_1, \dots, u_{r};
\lambda_{1\le j\le i\le r})\cup\j
L(c_1,\dots,c_{r}; u_1, \dots, u_{r};
\lambda_{1\le j\le i\le r})=0
$$
is expanded into the system of equations
$$
\text{($\dagger$)}
\begin{cases}
c_1'\cdot c_{r} + \lambda_{r1} C_{11} +\dots+
 \lambda_{rr} C_{1r}=0 \\
\dots \\
c_{r-1}'\cdot c_{r} + \lambda_{r1} C_{(r-1)1} +
\dots+ \lambda_{rr} C_{(r-1)r}=0 \\
|c_r|^2 + 2\lambda_{r1} C_{r1} + \dots +
 2 \lambda_{rr} C_{rr}
+ \lambda_{r1}^2+\dots+\lambda_{rr}^2=0
\end{cases}
$$
First consider the case where
the coefficient row vectors
$$
(C_{11}, \dots, C_{1r}), \dots
(C_{(r-1)1}, \dots, C_{(r-1)r})
$$
are linearly dependent.
Say $(C_{s1}, \dots, C_{sr})$
is a linear combination of previous row vectors.
By induction,
we can assume
$$
(c_1, \dots, c_{s-1}, c_{s+1}, \dots, c_{r-1}, c_r;
u_1, \dots, u_{r-1})
$$
is a good configuration up to index shuffling.
Because of the linear dependence,
unraveling the equations tells us that
$$
(c_1, \dots, c_{s-1}, c_{s+1}, \dots, c_{r-1}, c_r, c_s;
u_1, \dots, u_{r-1}, u_r)
$$
is automatically a good configuration.

Now we can assume the coefficient matrix
$$
C_{**}:=\begin{bmatrix}
C_{11} & \dots & C_{1r} \\
\dots & \dots & \dots \\
C_{(r-1)1} & \dots & C_{(r-1)r}
\end{bmatrix}
$$
has rank equal to $(r-1)$.
We can therefore do a full Gauss elimination
and convert
($\dagger$)
into a single variable quadratic equation
depending on some variable $\lambda=\lambda_{ri}$ ($1\le i\le r$).
So either ($\dagger$)
has finitely many solutions,
or the quadratic (that is, the last)
equation of ($\dagger$) can be factored
into a product of two linear equations
such that one of the linear factors
has coefficient vector
linearly dependent on
rows of the coefficient matrix $C_{**}$,
which forces
$$
|c_r|^2 + 2\lambda_{r1} C_{r1} + \dots +
 2 \lambda_{rr} C_{rr}
+ \lambda_{r1}^2+\dots+\lambda_{rr}^2
$$
to be a product of two linear factors.

\begin{lem}
\label{lem:mixed-term}
Any reducible quadratic polynomial in three variables $x, y, z$
with leading terms $x^2+y^2+z^2$
must have a mixed quadratic term.

In other words, 
any non-degenerate symmetric bilinear form
in more than $2$ variables
defines an irreducible hypersurface.
\end{lem}

\begin{proof}
Say $(x+ay+bz + e) (x+cy+dz+f) = x^2+y^2+z^2+$ linear part.
Expanding it we get
$$
\begin{cases}
ac=1 \\
a+c=0 \\
bd=1\\
b+d =0 \\
ad+bc=0
\end{cases}
$$
Since $a=-b$, $c=-d$,
we have $ad+bc=2bc\not=0$.
\end{proof}

The induction is now complete, and we have proved:

\begin{prop}
There exists a reordering of $c_1, \dots, c_k$
such that
$(c_1, \dots, c_k; u_1, \dots, u_k)$
is a good configuration.
\end{prop}

Indeed, our discussion can be used
to establish the following:

\begin{prop}
Let
$\lambda_{ij}$ ($1\le i, j\le k$)
be $k^2$ indeterminants.

Set $\wt c_i := c_i + \lambda_{i1}u_1 +\dots
\lambda_{ik}u_k$.
Set $\wt L:= \wt c_1 + \RSw_{1,2}\wt c_2 +
\dots+\RSw_{1,k}\wt c_k$.
Consider the equation ($\dagger$)
$$
\wt L \cup \j \wt L=0
$$
in $k^2$-variables.
Define $N$ to be the maximum number of variables 
in the system of equations ($\dagger$) that can
independently take any value in some infinite subset 
of $\bFp$
while the system remains solvable.

Then $N \le k(k-1)/2$.
\label{prop:color}
\end{prop}

\begin{proof}
We only need to take a closer look
at the earlier discussions.

Assign a color (white - allowed to take
infinitely many values independently, and black - otherwise)
to each variable in
$$
\begin{bmatrix}
\lambda_{11} & \dots & \lambda_{1k} \\
\dots & \dots & \dots \\
\lambda_{k1} & \dots & \lambda_{kk}
\end{bmatrix}
$$

By Step 1, no row can be fully white.
Without loss of generality, we assume $\lambda_{11}$
is black.

By Step 2, the blackness of $\lambda_{11}$
implies
either $\lambda_{12}$ and one of $(\lambda_{21}, \lambda_{22})$
is black,
or both of $(\lambda_{21}, \lambda_{22})$ are black.
Without loss of generality,
we assume both of $(\lambda_{21}, \lambda_{22})$ are black
by possibly exchanging $c_1$ and $c_2$.

We inductively demonstrate by reordering indices,
it is possible to guarantee
$\lambda_{ij}$ ($1\le j\le i \le r$)
are all black.

Going back to Step 3.
If $C_{**}$ is not full rank,
then a certain row $(\lambda_{s1}, \dots, \lambda_{sr})$ is entirely determined by
the other rows;
so by swaping the color of white $\lambda_{si}$
with the color of black variables on the other rows,
we can assume without loss of generality
that all of $(\lambda_{s1}, \dots, \lambda_{sr})$
are black.
On the other hand, if $C_{**}$ is indeed full rank,
Lemma
\ref{lem:mixed-term}
implies 
all of $(\lambda_{r1}, \dots, \lambda_{rr})$
must all be black.
So, we are done.
\end{proof}

\begin{cor}
\label{cor:cone-estimate}
$\dim \wt X_k - \dim \wt X_k^c \ge \frac{k(k+1)}{2}$.
\end{cor}

\begin{proof}
Let $Y$ be the orthogonal complement
of $u_i^j$ in $X$, $1\le i, j\le k$.
$\dim Y = \dim X - k^2$.

Since $c\in H^1(\Gal_F, S)$
and $\dim S=k$,
we can assume $c_{k+1}=c_{k+2}=\dots=c_m=0$.
By Proposition \ref{prop:color},
    the projection $\wt X_k^c \to Y \times \GR(k, m)$
has fibers of dimension at most $k(k-1)/2$.
So we are done; also see Lemma \ref{lem:fiber}.
\end{proof}

The following is obvious.
\begin{lem}
\label{lem:wtXk}
$\dim \wt X_k \le \dim \GR(k, m) + kh=k(m-k)+kh$.
\end{lem}

\begin{proof}
$\wt X_k$
can be Zariski locally embedded in
$\GR(k, m) \times H^1(\Gal_F, S)$
and $\dim H^1(\Gal_F, S)= kh$.
\end{proof}

\begin{thm}[The cone theorem, I]
\label{thm:cone}
We have
$$
\dim X^c \le
\max_{0\le k \le m, h}
k(m-k) + kh - \frac{k(k+1)}{2}.
$$
Equivalently,
$$
\dim X - \dim X^c \ge
\min_{0\le k \le m, h}
(h-k)(m-k) + \frac{k(k+1)}{2}.
$$
\end{thm}

\begin{proof}
Simply combine Corollary \ref{cor:cone-estimate}
and Lemma \ref{lem:wtXk}.
\end{proof}

Set $\delta(m, h;k):=
(h-k)(m-k) + \frac{k(k+1)}{2}$ ($0\le k\le m,h$).
The critical point of $\delta(m,h;k)$
is $k=\frac{m+h}{3}-1/6$.
We temporarily suppose $m\le h$
because of the symmetry of $m$ and $h$
in the expression of $\delta(m,h;k)$.
We have
$$
\delta(m, h;k)_{\min}
\ge
\begin{cases}
\delta(m, h;\frac{m+h}{3}-1/6) & m \ge h/2-1/4 \\
\delta(m, h; m) & m \le h/2-1/4
\end{cases}
$$
The partial derivative
$$
\frac{\partial}{\partial h}\delta(m, h;\frac{m+h}{3}-1/6)
=-h/3+2m/3+1/6\ge 0,
$$
and thus
$\delta(m, h;\frac{m+h}{3}-1/6)$
attains minimum when $m=h$
and the minimum value is
$$
\frac{1}{3}(m^2+m)-1/24.
$$
but since $\delta(m,h;k)$ is an integer,
its minimum value is at least $(m^2+m)/3$.
On the other hand
$D(m,h;m) = (m^2+m)/2$.
Combining both cases, we have

\begin{cor}
\label{cor:cone}
(1) We have
$$
\dim X - \dim X^c \ge
\frac{1}{3}(m^2+m).
$$

(2)
We have
$
\dim X - \dim X^c \ge
m.
$
\end{cor}

\begin{proof}
(1) is already proved.
(2) follows from (1).
\end{proof}

When $m=2$, we do need a more accurate estimate:

\begin{cor}
\label{cor:cone-m2}
If $m=2$,
then $\dim X-\dim X^c \ge \min(3, 1+\dim V_m/m)$.
\end{cor}

\begin{proof}
Enumerating the cases $k=0,1,2$,
we get
$\dim X-\dim X^c \ge \min(3, h, 2h)$.
But $h\ge 1+\dim V_m/m$ by Euler characteristic.
\end{proof}

\section{Reduced Emerton-Gee stacks}

In this part, we establish some useful results for computing the dimension
of the reduced Emerton-Gee stacks.
Write $\cX_{\lsup LM}=\cX_{\lsup LM, F}$
for the {\it reduced} Emerton-Gee stacks
over $\bFp$.
By \cite{L23B}, all reduced Emerton-Gee stacks
are of finite type,
and their $\bFp$-points
form a groupoid
equivalent to the groupoid of Langlands parameters
with $\bFp$-coefficients
-- and that's basically all we need to know
about the Emerton-Gee stacks.

We suppress subscripts, and set
$G = G_n,~
P = P_m,\text{~and,~}
M = M_m$.

Since $M_m = \Res_{K/F}\GL_m\times G_{n-2m}$,
there is a morphism
$\cX_{\lsup LM} \to \cX_{\lsup L\Res_{K/F}\GL_m}$.
Denote by $\cX_{\lsup LM}'\subset \cX_{\lsup LM}$
the closed substack
whose image in $\cX_{\lsup L\Res_{K/F}\GL_m}$
consists of irreducible Galois representations.

Denote by $\vec{h}=(v_0,v_1,v_2,z_0,z_1,z_2)$
a tuple of 6 integers.
The condition
\begin{eqnarray*}
\dim H^i(\Gal_{F}, Z)&=z_i \\
\dim H^i(\Gal_{F}, V)&=v_i
\end{eqnarray*}
cuts a locally closed substack 
$\cX_{\lsup LM}^{\vec h}$
of $\cX_{\lsup LM}'$
because Galois cohomology can be algebraically interpolated
by the Herr complexes;
and by the semicontinuity theorem,
the locus where a perfect complex
has fixed Poincar\'e polynomial is locally closed.

\begin{lem}
\label{lem:fiber}
Let $\cX\to\cY$ be a surjective morphism
of algebraic stacks of finite type over $\bFp$.

(1)
If for all $x:\spec \bFp\to \cY$,
the fiber $\cX\times_{\cY, x}\spec \bFp$
has dimension at most $d$,
then $\dim \cX \le \dim \cY + d$.

(2)
If for all $x:\spec \bFp\to \cY$,
the fiber $\cX\times_{\cY, x}\spec \bFp$
has dimension at least $d$,
then $\dim \cX \ge \dim \cY + d$.
\end{lem}

\begin{proof}
(1)
It follows from \cite[Tag 02JU, 0AFN, 0AFP]{Stacks}.

(2)
It follows from \cite[Tag 02JU, 02JB, 0AFN, 0AFP]{Stacks}.
\end{proof}

\begin{cor}
\label{cor:relative-dim}
(1)
Let $\lsup LH \hookrightarrow \lsup LH'$ be an $L$-embedding
between algebraic groups.
If $\cX\subset \cX_{\lsup LH}$
is an algebraic substack
and $\cY$ is the scheme-theoretic image of $\cX$
in $\cX_{\lsup LH'}$,
then $\dim \cY \le \dim \cX$.

(2)
Let $\lsup LH \to \lsup LH'$ be an isogeny
between algebraic groups.
If $\cX\subset \cX_{\lsup LH}$
is an algebraic substack
and $\cY$ is the scheme-theoretic image of $\cX$
in $\cX_{\lsup LH'}$,
then $\dim \cY = \dim \cX$.
\end{cor}

\begin{proof}
(1)
Let $x:\spec\bFp \to \cX_{\lsup LH'}$.
There exists a monomorphism
$
[\Aut_{\breve{H}'}(x) /\Aut_{\breve H}(x)]
\hookrightarrow
\cX_{\lsup LH}\times_{\cX_{\lsup LH'}, x}\spec \bFp$.
Note that $\Aut_{\breve{H}}(x) \subset \Aut_{\breve H'}(x)$
is a subgroup.

(2)
All fibers 
$\cX_{\lsup LH}\times_{\cX_{\lsup LH'}, x}\spec \bFp$
have finitely many $\bFp$-points
and are non-empty,
and thus $0$-dimensional.
\end{proof}

\subsection{Coarse Moduli Spaces}
If $\cX$ is an algebraic stack,
define the presheaf of sets
$|\cX|: T\mapsto |\cX(T)|$
where $|\cX(T)|$ is the the set
of equivalence classes of the groupoid $\cX(T)$.
We say $|\cX|$ is the coarse moduli space
of $\cX$ is it is representable by algebraic spaces.

\section{Involutions and elliptic parameters}

Starting from this section the subscript $n$ in $G_n$
will change,
and as a consequence, we use $M_{m,n}=M_m$
to denote the $m$-th standard Levi.

For Galois representations $\rho$
valued in a similitude group,
we denote by $\lambda_\rho$ its similitude character.

A paramater is said to be {\it elliptic}
if it does not factor through any proper parabolics.

\begin{dfn}
\label{def:theta}
Define $\theta:\cX_{\lsup LM_{m, n+2m}} \to \cX_{\lsup LM_{m, n+2m}}$
by
$$
\theta(\psi)=
\begin{cases}
\psi(\j\circ \j^{-1})^{-t} & G\text{~unitary} \\
\lambda_{\psi}
\begin{bmatrix}
& & w_m \\
& w_n' & \\
-w_m & &
\end{bmatrix}
\psi^{-t} 
\begin{bmatrix}
& & w_m \\
& w_n' & \\
-w_m & &
\end{bmatrix}^{-1}
& G\text{~symplectic} \\
\lambda_{\psi} w_{n+2m} \psi^{-t} 
w_{n+2m}
& G\text{~orthogonal}
\end{cases}
$$
for $\psi:\Gal_F\to \lsup LM_m(\bFp)$;
$\theta$ is an involution in the sense that $\theta\circ \theta=1$.

Denote by $\cX_{\lsup LG}^{\Ell}$ the closed substack
of $\cX_{\lsup LG}$ formed by elliptic parameters.

An $\bFp$-point of $\cX_{\lsup LM_{m,n+2m}}^{\Ell}$
is represented by a pair $(\alpha, \rho)$
where $\alpha:\Gal_F\to\lsup L\Res_{K/F}\GL_m(\bFp)$
is an irreducible Galois representation
and $\rho\in \cX_{\lsup LG_n}^{\Ell}$
is elliptic.
Set $(\theta(\alpha), \theta(\rho)):=\theta(\alpha, \rho)$.
Also denote by 
$\alpha_K:\Gal_K\to \GL_m(\bFp)$
the Galois representation
that corresponds to $\alpha$
via non-abelian Shapiro lemma.
\end{dfn}

\begin{lem}
If $v_2>0$, then
$\cX_{\lsup LM_{m,n+2m}}^{\vec h, \Ell}$
is empty over the non-self-adjoint locus
where $\alpha(-1)\ncong \theta(\alpha(-1))$.
\label{lem:nsad}
\end{lem}

\begin{proof}
First consider the case
where $G$ is either symplectic or orthogonal:
$\alpha(-1)^{\oplus v_2}$
is a subrepresentations of $\rho$;
if $\alpha(-1)^{\oplus v_2}$
is not an isotropic subspace of the repsesentation space
of $\rho$,
then there exists $x, y$ in the representation space
of $\alpha(-1)^{\oplus v_2}$ such that $(x,y)\neq 0$;
for any basis $\{e_i\}$ of $\bFp[\Gal_F]x$
and the basis $\{e_j^*\}$ of $\bFp[\Gal_F]y$
dual to $\{e_i\}$,
we have
$$
(ge_i, ge_j^*)=\lambda(g)(e_i, e_j^*)=\lambda\delta_{i,j}
$$
for all $i, j$ and $g\in\Gal_F$;
here $\lambda$ is the similitude character.
It implies
$\alpha(-1)\cong \bFp[\Gal_F]y\cong \theta(\alpha(-1))$.

Then consider the unitary case:
$Q:=\alpha(-1)^{\oplus v_2}_K$
is a subrepresentations of $\rho|_{\Gal_K}$;
and
$\j Q\oplus Q$ is a subrepresentation of $\rho$.
Denote by $Q'$ the sum of all irreducible subrepresentations
of $\rho|_{\Gal_K}$ that does not contain
isomorphic copies of 
$\alpha(-1)$ or $\theta(\alpha(-1))$.
Then $\rho$ factors through the proper parabolic
$
\begin{bmatrix}
\GL(Q) & * & * \\
& \GL(Q') & * \\
& & \GL(\j Q)
\end{bmatrix}\rtimes \Gal(K/F).
$
\end{proof}

Let $(\alpha, \rho)\in \cX_{\lsup LM_{m,n+2m}}^{\vec h, \Ell}$
be such that
$\alpha(-1)\cong \theta(\alpha(-1))$.
$\rho_0:=\alpha(-1)^{\oplus v_2}$
is a subrepresentation of $\rho$.
Denote by $\rho_1$ the sum
of irreducible $\Gal_K$-subrepresentation
of $\rho$ not isomorphic to $\alpha(-1)$.

We have
$
\Img\rho_0 
\subset
\lsup LG_{v_2m}
$
and
$
\Img\rho_1 
\subset
\lsup LG_{n-v_2m}
$,
and they have the same similitude character.

\begin{lem}
\label{lem:dim-ell}
We have 
$
\dim \cX_{\lsup LM_{m, n+2m}}^{\vec h, \Ell}
\le
-1
-\dim \Aut_{\bG_{v_2m}^{\ad}}(\rho_0).
$
\end{lem}

\begin{proof}
Note that
$\rho_1$ factors through
$\lsup LS\subset {\lsup LG}_{n-v_2m}$
where $S$ is a torus (\cite{L23}).
There are only finitely choices
of $\bG_{n-v_2m}$-conjugacy classes of $\lsup LS$;
so we can fix $S$.

First consider the unitary case.
Self-adjointness
determines the determinant character of $\alpha$
up to finite possibilities
(since $\alpha(-1)(\j \gamma \j) \cong \alpha(-1)(\gamma)^{-t}$);
and irreducible $\alpha$
has only finitely many possibilities up to
determinant characters.
So $\alpha$ has only finitely many possibilities.
We can thus fix $\alpha$.
The parameter $(\alpha, \rho)$
factors through 
$\lsup L\Res_{K/F}\GL_m\times \lsup LG_{v_2m} \times \lsup LS$.
By
Corollary \ref{cor:relative-dim},
\begin{align*}
\dim \cX_{\lsup LM_{m, n+2m}}^{\vec h, \Ell}
&\le 
\dim [\spec \bFp/\Aut_{\breve{\Res_{K/F}\GL_m}}{\alpha}]
\times
[\spec \bFp/\Aut_{\bG_{v_2m}}(\rho_0)]
\times \cX_{\lsup LS}
\\
&=
-1 - \dim \Aut_{\bG_{v_2m}}(\rho_0) + 0.
\end{align*}

Next consider the symplectic / orthogonal case.
Consider the morphism
$$
(\alpha, \rho)\mapsto \rho_1: 
\cX_{\lsup LM_{m, n+2m}}^{\vec h, \Ell} \to 
\cX_{\lsup LS}.
$$
We know $\dim \cX_{\lsup LS}=0$
and thus the dimension of 
$\cX_{\lsup LM_{m, n+2m}}$
is the same as the dimension
of a non-empty fiber.
Along each fiber, 
$\rho_1$ determines the similitude character
of $\rho_0$
and thus determines $\rho_0$
up to finitely many choices;
and $\rho_0$ determines $\alpha$.
So the dimension of each non-empty fiber is
$$
[\bS / \Aut_{\GL_m \times 
\bS\bG_{v_2m}}(\alpha, \rho_1, \rho_0)]
=
-1 - \dim \Aut_{\bG_{v_2m}^{\ad}}(\rho_0),
$$
since
$\bS\bG_{v_2m}/\bS = \bG_{v_2m}^{\ad}$.
The lemma now follows from Lemma 
\ref{lem:aut}.
\end{proof}

\begin{lem}
\label{lem:aut}
We have
$$
    \Aut_{\bG_{v_2m}}(\rho_0)
\cong
\begin{cases}
\GSp_{v_2} & G\text{~symplectic}\\
\GO_{v_2} & G\text{~orthogonal}\\
\Or_{v_2} & G\text{~unitary}.
\end{cases}
$$
\end{lem}

\begin{proof}
Since $\alpha$ is irreducible,
we can assume without loss of generality $m=1$
by Schur's lemma.
In the symplectic or orthogonal case,
$\rho_0$ is a scalar matrix
$
\begin{bmatrix}
\alpha(-1) & & \\
& \dots & \\
& & \alpha(-1)
\end{bmatrix}
$
and
its automorphism group is the entire $\bG_{v_2}$.

Next consider the unitary case:
for $\gamma\in \Gal_F\backslash \Gal_K$,
we have
$$\rho_0(\gamma)
=
\begin{bmatrix}
\alpha(-1)_K(\gamma) & & \\
& \dots & \\
& & \alpha(-1)_K(\gamma)
\end{bmatrix}
w_{v_2}\rtimes \j$$
where $w_{v_2}$ is the longest Weyl group element for
$\GL_{v_2}$
(up to isomorphism,
$w_0$ is the only constant matrix that makes $\rho_0$ elliptic).
Note that
$(w_{v_2}\rtimes\j)g(w_{v_2}\rtimes \j)=g^{-t}$.
So $g\in \Aut_{\bG_{v_2}}(\rho_0)$
if and only if $gg^{t}=1$.
\end{proof}

\section{Universal induction setup}

\begin{lem}
\label{lem:sad-nsad}
Let $(\alpha, \rho)\in \cX_{\lsup LM_{m,n+2m}}^{\vec h}$.
Recall that $\alpha$ is irreducible.

(1)
If $\alpha(-1) \not\cong \theta(\alpha(-1))$,
then $\rho$ factors through the
parabolic subgroup
$\lsup LP_{v_2m,n}$
of $\lsup LG_n$:
indeed
we have matrix presentation
$$
\rho
=
\begin{bmatrix}
\alpha(-1)^{\oplus v_2} & * & *\\
& * & * \\
& & \theta(\alpha(-1)^{\oplus v_2})
\end{bmatrix}\rtimes *
\subset
\begin{bmatrix}
\GL_{v_2m} & * & * \\
& \breve{G}_{n-2v_2m}  & * \\
& & \GL_{v_2m}
\end{bmatrix}\rtimes \Gal(K/F).
$$
In otherwords, $\rho$
has a totally isotropic subrepresentation
isomorphic to
$\alpha(-1)^{\oplus v_2}$.

(2)
If $\alpha(-1) \cong \theta(\alpha(-1))$,
then 
we can write
$v_2 = v_2^a + v_2^i$
such that
$\rho$ has a totally isotropic subrepresentation
isomorphic to 
$\alpha(-1)^{\oplus v_2^i}$
and a totally anisotropic subrepresentation isomorphic
to 
$\alpha(-1)^{\oplus v_2^a}$
in the sense that:
$\rho$ factors through
$\lsup LP_{v_2^im, n-v_2^am}\lsup LG_{v_2^am}$
and we have matrix presentation
$$
\rho
=
\begin{bmatrix}
\alpha(-1)^{\oplus v_2^i} & * & * & 0\\
& * & * & 0\\
& & \theta(\alpha(-1)^{\oplus v_2^i}) & 0\\
& & & \rho^a
\end{bmatrix}\rtimes *
\subset
\begin{bmatrix}
\GL_{v_2^im} & * & * &0\\
& \breve{G}_{n-2v_2^im-v_2^am}  & *&0 \\
& & \GL_{v_2^im}&0\\
& & & \breve{G}_{v_2^am}
\end{bmatrix}\rtimes \Gal(K/F).
$$
Here $\rho^a$ is a direct sum of elliptic parameters.
\end{lem}

\begin{proof}
(1)
It is already proved in the proof of Lemma \ref{lem:nsad}.

(2)
Only need to note that you can take the symplectic (or orthogonal or unitary)
complement of an anisotropic subrepresentation.
\end{proof}

Consider the diagram
$$
\xymatrix{
  \cY^{\vec h} \ar[r] \ar[d]^f & \cX^{\vec h}_{\lsup LM_{m, n+2m}} \\
  \cX_{\lsup LM_{m, n-2v_2m+2m}}
}
$$
where $\cY^{\vec h}\subset 
\cX_{\lsup L\Res_{K/F}\GL_m\times \lsup LP_{v_2m,n}}
\underset{\cX_{\lsup LM_{m,n+2m}}}{\times}
\cX_{\lsup LM_{m,n+2m}}^{\vec h}
$
consists of pairs $(\alpha, \rho)$
where 
$$
\rho =
\begin{bmatrix}
\alpha(-1)^{\oplus v_2} & * & *\\
& \tau & * \\
& & \theta(\alpha(-1)^{\oplus v_2})
\end{bmatrix}\rtimes *
\subset
\begin{bmatrix}
\GL_{v_2m} & * & * \\
& \breve{G}_{n-2v_2m}  & * \\
& & \GL_{v_2m}
\end{bmatrix}\rtimes \Gal(K/F),
$$
and
$f(\alpha, \rho) = (\alpha(-1), \tau)$.

\begin{dfn}
Denote by $\cX^{\vec h, \circ}_{\lsup LM_{m, n+2m}}$
the scheme-theoretic image
of $\cY^{\vec h}$
in $\cX^{\vec h}_{\lsup LM_{m, n+2m}}$.
\end{dfn}

We remark,
without using or proving the statement, that
 $\cX^{\vec h, \circ}_{\lsup LM_{m, n+2m}}$
is dense in $\cX^{\vec h}_{\lsup LM_{m, n+2m}}$;
this is not surprising 
because by part (2)
of Lemma \ref{lem:sad-nsad},
the complement of 
 $\cX^{\vec h, \circ}_{\lsup LM_{m, n+2m}}$
consists parameters
whose entire rows consists of zero entries.

\begin{lem}
\label{cor:dense}
$\dim \cX^{\vec h, \circ}_{\lsup LM_{m, n+2m}} \le \dim \cY^{\vec h}$.
\end{lem}

\begin{proof}
It follows from Corollary 
\ref{cor:relative-dim}.
\end{proof}

Let $x: \spec \bFp \to \cX_{\lsup LM_{m, n-2v_2m+2m}}^{\vec h'}$
be a point
where $\vec h' = (v_0', v_1', v_2', z_0', z_1', z_2')$,
that corresponds to the parameter $\tau$.
We will analyze the dimension of the fiber
$\cY_x^{\vec h}\to \spec\bFp$
 of $f$ at $x$.
The point $x$ factors through
$$
\spec \bFp \to [\spec \bFp/\Aut(x)]
\to \cY^{\vec h},
$$
where
$$
\Aut(x) := \Aut_{\GL_m}(\alpha)
\times
\Aut_{\bG_{n-2v_2m}}(\tau)
$$
So, $\Aut(x)$ acts on 
$\cY_x^{\vec h}$
by pullback;
and we can form the quotient
$[\cY_x^{\vec h}/\Aut(x)]$.

The $\bFp$-points of the coarse moduli space
$|\cY_x^{\vec h}|$
are tautologically identified with
the cohomology set $H^1(\Gal_F, U_{v_2m, n})$.
There exists a long exact sequence
$$
H^1(\Gal_F, Z_{v_2m, n})
\to
H^1(\Gal_F, U_{v_2m, n})
\to 
H^1(\Gal_F, V_{v_2m, n}).
$$
We have already that seen the image of
$H^1(\Gal_F, U_{v_2m, n})$
in
$H^1(\Gal_F, V_{v_2m, n})=:X$
is an affine cone $X^c$ defined by the kernel
of cup product.
All elements of $c\in X$
corresponds to a parameter
$(\alpha, c\rho):\Gal_F \to 
\lsup L\Res_{K/F}\GL_m\lsup LP_{v_2m, n}$,
whose
automorphism group
is equal to
$$
H^0(\Gal_F, V_{v_2m, n}) 
\rtimes (\GL_{v_2} \times \Aut(x)),
$$
so there is a surjective morphism
$$
[\cY_x^{\vec h}/\Aut(x)]
\to |\cY_x^{\vec h}| 
\to X^c \to [X^c/
H^0(\Gal_F, V_{v_2m, n}) 
\rtimes (\GL_{v_2} \times \Aut(x))
];
$$
whose fibers 
are all equal to
$[H^1(\Gal_F, Z_{v_2m, n})/H^0(\Gal_F, Z_{v_2m, n})]$
(the reader can check the fibers
by unravelling the definitions
of quotient stacks and fiber product stacks).

\begin{dfn}
Define $\delta_x := \dim X - \dim X^c \ge 0$.
\end{dfn}

So
\begin{align*}
\dim \cY^{\vec h}_x - \cancel{\dim \Aut(x)}
=&
\boxed{\dim H^1(\Gal_F, Z_{v_2m, n})
-
\dim H^0(\Gal_F, Z_{v_2m, n}) }
+ \\
& \dim X^c - (\dim H^0(\Gal_F, V_{v_2m, n})  +
v_2^2 + \cancel{\dim \Aut(x)})
\end{align*}

By the Euler characteristic,
we have
\begin{align*}
\boxed{\dim H^1(\Gal_F, V_{v_2m, n})
-\dim H^0(\Gal_F, V_{v_2m, n})}
&= v_2v_2' + 
[F:\bQ_p]\dim V_{v2_m, n}\text{,~and}
\\
\dim H^1(\Gal_F, Z_{v_2m, n})
-
\dim H^0(\Gal_F, Z_{v_2m, n})
&=
\dim H^2(\Gal_F, Z_{v_2m, n})
+
[F:\bQ_p]\dim Z_{v_2m, n}
\end{align*}
Note that $\dim H^2(\Gal_F, V_{v_2m, n})=v_2v_2'$.

\begin{lem}
\label{lem:Yx}
(1)
We have
\begin{align*}
\dim \cY^{\vec h}_x
&= v_2v_2' + [F:\bQ_p]\dim V_{v_2m,n}
-\delta_x + H^2(\Gal_F, Z_{v_2m,n}) + [F:\bQ_p]\dim Z_{v_2m, n}
-v_2^2
\\
&=
v_2v_2' - v_2^2
-\delta_x
+[F:\bQ_p]\dim U_{v_2m, n}
+ H^2(\Gal_F, Z_{v_2m,n}).
\end{align*}

(2)
$\dim H^2(\Gal_F, Z_{v_2m,n})=0$
if $\alpha(-2)\not\cong \theta(\alpha(-1))$,
and
$$
\dim H^2(\Gal_F, Z_{v_2m,n})=
\begin{cases}
\frac{v_2(v_2+1)}{2} & \text{symplectic or unitary} \\
\frac{v_2(v_2-1)}{2} & \text{orthogonal}.
\end{cases}
$$
if otherwise.
\end{lem}

\begin{proof}
Note that $\dim X^c = \dim X - \delta_x
=\dim H^1(\Gal_F, V_{v_2m, n})-\delta_x$.
\end{proof}

\begin{prop}
\label{prop:X-ind}
We have
\begin{align*}
\dim \cX_{\lsup LM_{m, n+2m}}^{\vec h, \circ}
\le \max_{x\in \cX_{\lsup LM_{m, n-2v_2m+2m}}^{
\vec h'}(\bFp)} &
\dim \cX_{\lsup LM_{m, n-2v_2m+2m}}^{\vec h'} + 
\dim \cY_x^{\vec h}
\end{align*}
\end{prop}

\begin{proof}
It follows from Lemma \ref{lem:fiber} and
Lemma \ref{cor:dense}.
\end{proof}

Let $u_n$
be the dimension
of the unipotent radical of 
a Borel of $\lsup L\Res_{F/\bQ_p}G_{n}$.
It is expected that
$\dim \cX_{\lsup LM_{m, n+2m}}^{\vec h}$
should be slightly smaller than $u_{n}$.

We record the value of $u_n$:

\begin{dfn}
\label{lem:val-u}
Define
$
u_{n}
=
[F:\bQ_p]
\begin{cases}
n^2/4 & \text{symplectic} \\
\lceil n^2/4-n/2\rceil & \text{orthogonal}\\
n(n-1)/2 & \text{unitary}.
\end{cases}
$
\end{dfn}

\begin{dfn}
Define 
\begin{align*}
\delta_{m, n+2m}^{\vec h}&:=u_{n} -
 \dim\cX_{\lsup LM_{m, n+2m}}^{\vec h},\\
\delta_{m, n+2m}^{\vec h, \circ}&:=u_{n} -
 \dim\cX_{\lsup LM_{m, n+2m}}^{\vec h, \circ}.
\end{align*}
\end{dfn}

Note that
$$
u_{n} - u_{n-2v_2m}=
[F:\bQ_p](
\frac{v_2m(v_2m-1)}{2}+
\dim U_{v_2m, n}).
$$
Thus,

\begin{prop}
We have
\label{prop:universal}
\begin{align*}
\delta_{m, n+2m}^{\vec h, \circ}
\ge& \min_{x\in \cX_{\lsup LM_{m, n-2v_2m+2m}}^{\vec h'}(\bFp)}
\delta_{m, n-2v_2m+2m}^{\vec h'} + \delta_x
-v_2v_2'+v_2^2\\
&+[F:\bQ_p]\frac{v_2m(v_2m-1)}{2}
-\dim H^2(\Gal_F, Z_{v_2m, n}).
\end{align*}
\end{prop}

\begin{proof}
Combine Lemma \ref{lem:Yx}
and Proposition \ref{prop:X-ind}.
\end{proof}

Note that
$$
[F:\bQ_p]v_2m(v_2m-1) - v_2(v_2-1) \ge v_2m(m-1),
$$
and observe that
$$
-v_2v_2'+v_2^2=\frac{1}{2}(v_2^2-v_2^{\prime2}) + \frac{1}{2}
(v_2-v_2')^2
\ge \frac{1}{2}(v_2^2-v_2^{\prime2})
$$
we get a slightly weaker inequality

\begin{cor}
We have
\label{cor:universal}
\begin{align*}
\delta_{m, n+2m}^{\vec h, \circ}
-\frac{1}{2}v_2^2
\ge& \min_{x\in \cX_{\lsup LM_{m, n-2v_2m+2m}}^{\vec h'}(\bFp)}
\delta_{m, n-2v_2m+2m}^{\vec h'} 
-\frac{1}{2}v_2^{\prime2}
\\
&+\delta_x+\frac{v_2(v_2-1)}{2}
-\dim H^2(\Gal_F, Z_{v_2m, n})
+ \frac{m(m-1)}{2}v_2.
\end{align*}
\end{cor}

\begin{prop}
\label{prop:deltap}
We have
$$
\delta_x+\frac{v_2(v_2-1)}{2}
-\dim H^2(\Gal_F, Z_{v_2m, n})+
\frac{m(m-1)}{2}v_2
\ge
\min(\delta_x, \frac{m(m-1)}{2}v_2)
\ge 0.
$$
\end{prop}

\begin{proof}
It is clear unless 
$\dim H^2(\Gal_F, Z_{v_2m,n})>0$,
so we assume that.
Recall that
$$
\dim H^2(\Gal_F, Z_{v_2m,n})=
\begin{cases}
\frac{v_2(v_2+1)}{2} & \text{symplectic or unitary} \\
\frac{v_2(v_2-1)}{2} & \text{orthogonal}.
\end{cases}
$$
When $m=1$,
we only need to show
$$
\delta_x
\ge
\begin{cases}
v_2 & \text{symplectic or unitary} \\
0 & \text{orthogonal},
\end{cases}
$$
which follows from Corollary \ref{cor:cone};
when $m>1$, it follows from
$m(m-1)/2\ge 1$.
\end{proof}

Combine
Proposition
\ref{prop:deltap}
and Corollary 
\ref{cor:universal},
we see that
\begin{align}
\delta_{m, n+2m}^{\vec h, \circ}
-\frac{1}{2}v_2^2
\ge& \min_{x\in \cX_{\lsup LM_{m, n-2v_2m+2m}}^{\vec h'}(\bFp)}
\delta_{m, n-2v_2m+2m}^{\vec h'} 
-\frac{1}{2}v_2^{\prime2}.
\label{eq:universal}
\end{align}

\section{Dimension calculation}

\begin{lem}
\label{lem:degen}
Fix $m$.
Let $P(n)$ be the statement that
$$\delta^{\vec h}_{m, n+2m}\ge v_2^2/2$$
and let  $P^\circ(n)$ be the statement that
$$\delta^{\vec h, \circ}_{m, n+2m}\ge v_2^2/2.$$
We have 
$$P(<n) \Rightarrow (P(n) \Leftrightarrow P^\circ(n)).$$
\end{lem}

We will prove it later.
As a consequence of Lemma \ref{lem:degen},
the reader can pretend $\delta^{\vec h}_{m, n+2m}=
\delta^{\vec h, \circ}_{m, n+2m}$
in this section.

\begin{thm}
\label{thm:universal}
(1) 
If $G$ is orthogonal, then
$\delta_{m, n+2m}^{\vec h} \ge \frac{1}{2}v_2^2+
\min_x\delta_x.$

(2) If $G$ is either symplectic or unitary, then
$\delta_{m, n+2m}^{\vec h} \ge \frac{1}{2}v_2^2$.

(3) If $G$ is either symplectic or unitary
and
moreover $v_2=2$ and $m=1$,
then
$\delta_{m, n+2m}^{\vec h} \ge 3$.
\end{thm}

\begin{proof}
(1)
By (Eq. \ref{eq:universal}),
it remains to establish the base case of the induction:
we need to show
$u_{n}-\dim \cX^{\Ell,\vec h}_{m, n+2m}\ge v_2^2/2$.
By Lemma \ref{lem:dim-ell} and Lemma \ref{lem:aut},
it suffices to show
$$
u_{n} + 1 + \dim
\begin{cases}
\Sp_{v_2} & \text{symplectic} \\
\SO_{v_2} & \text{orthogonal}
\end{cases}
\ge \frac{1}{2}v_2^2
$$
By Lemma \ref{lem:val-u},
we have
$u_{n}\ge v_2^2-v_2$ (since $n\ge 2v_2$).
So it suffices to check
$v_2^2-v_2+1\ge v_2^2/2$, which is clear.

The extra enhancement term $\min_x\delta_x$
is obtained by using the original
Corollary \ref{cor:universal}
in the last step of induction.

(2)
It is the same as part (1)
except that we don't get the extra enhancement term.

(3)
Using Corollary \ref{cor:cone-m2}
for this corner case,
we get
$\delta^{\vec h}_{1, n+2}\ge \min(3, 1+n-2)$.
So we are done if $n\ge 4$.
But since $v_2=2$, $n\ge4$ for $\delta^{\vec h}_{1, n+2}$
to be even defined.
\end{proof}

\begin{cor}
\label{cor:extra-m}
We have $\delta^{\vec h}_{m, n+2m}> \frac{1}{2}v_2^2$
when $m>1$.
\end{cor}

\begin{proof}
By Proposition \ref{prop:deltap},
it suffices to show
$\delta_x>0$
when $\dim H^2(\Gal_F, Z_{v_2m, n})>0$.
If $\delta_x=0$,
then $X=X^c$ which is impossible
by local Tate duality.
\end{proof}

\begin{cor}
\label{cor:main}
$\dim \cX_{\lsup LG_n} = u_{n}$.
\end{cor}

\begin{proof}
By \cite{BG19},
for each regular Hodge type $\lambda$,
the de Rham lifting ring $\spec R$
of Hodge type $\lambda$ and any inertial type
is $p$-flat of
dimension $u_{n}+1+\dim \bG_n$.
So, $\spec R \otimes \bFp$
has dimension $u_{n}+\dim \breve {G}_{n}$.
Since $[\spec R \otimes \bFp/\breve {G}_{n}]_{\red}
\subset \cX_{\lsup LG_{n}}$,
we have
$\dim \cX_{\lsup LG_{n}} \ge u_{n}$.

Conversely,
consider the diagram
$$
\xymatrix{
& \cX_{\lsup L\Res_{K/F}\GL_m}^{\Ell}\times \cX_{\lsup LG_n} \ar@{^{(}->}[d] \ar@{-->}[ld] \\
\bigcup_{\vec h}\cX_{\lsup LM_{m, n+2m}}^{\vec h}
\ar[r]
&
\cX_{\lsup LM_{m, n+2m}}
}
$$
where $\cX_{\lsup L\Res_{K/F}\GL_m}^{\Ell}$
is the locus of irreducible Galois representations
and is of dimension $0$.
Note that $\cX_{\lsup LM_{m, n+2m}}^{\vec h}$
is empty except for finitely many $\vec h$.
It is clear that the vertical embedding
factors through the union.
So $\dim \cX_{\lsup LG_n}\le
 \max_{\vec h}\dim \cX_{\lsup LM_{m, n+2m}}^{\vec h}
\le u_n$.
\end{proof}

\subsection{The proof of Lemma \ref{lem:degen}}
\label{subsec:finish}
We still need to have a proof for Lemma \ref{lem:degen}.

\begin{proof}
By part (2) of Lemma \ref{lem:sad-nsad},
write $v_2=v_2^a+v_2^i$.
It suffices to show
$$
\frac{1}{2}v_2^2
\le
\frac{1}{2}(v_2^i)^2
+
\frac{1}{2}(v_2^a)^2
+ u_n - u_{n-v_2^am} - u_{v_2^am}
$$
which is clear using $n\ge 2v_2^i+v_2^a$.
\end{proof}

\section{Top-dimensional irreducible components}

Denote by $\lsup LB_n= \brB_n\rtimes \Gal(K/F)$
a Borel of $\lsup LG_n$
and let $\lsup LT_n=\bT_n\rtimes \Gal(K/F)\subset \lsup LB_n$
be a maximal torus.

\begin{lem}
\label{lem:image-P1}
The scheme-theoretic image of
the natural map
$f:\cX_{\lsup LP_{1,n}} \to \cX_{\lsup LG_n}$
contains all top-dimensional components of
 $\cX_{\lsup LG_n}$.
\end{lem}

\begin{proof}
Consider
$$
\bigcup_{m>1} \cX^{\vec h}_{\lsup LM_{m, n+2m}},
$$
which has dimension strictly smaller than
$u_n$ by Corollary \ref{cor:extra-m},
and its image in $\cX_{\lsup LG_n}$
consists of all parameters
that have an isotropic subspace of dimension
greater than $1$
(parameters
in the complement of the image of $\cX_{\lsup LP_{1,n}}$).
\end{proof}

Next, we analyze the structure of $\cX_{\lsup LP_{1,n}}$.
Recall that $U_{1,n}=Z_{1,n}\rtimes V_{1,n}$ is the unipotent radical
of $\lsup LP_{1,n}$.
Denote by
$$
\cX_{\lsup LP_{1,n}}^{(v_2;z_2)},~
\cX_{\lsup LM_{1,n}}^{(v_2;z_2)},~
\cX_{V_{1,n\rtimes}\lsup LM_{1,n}}^{(v_2;z_2)}
$$
the locally closed substacks
of 
$\cX_{\lsup LP_{1,n}},~
\cX_{\lsup LM_{1,n}},~
\cX_{V_{1,n\rtimes}\lsup LM_{1,n}}$ cut out by
the condition
$\dim H^2(\Gal_F, V_{1,n})=v_2$
and
$\dim H^2(\Gal_F, Z_{1,n})=z_2$.

\subsection{The case of orthogonal groups}
$Z_{1,n}$ is indeed trivial; so we always have $z_2=0$.
All fibers of
$
\cX_{\lsup LP_{1,n}}^{(v_2;z_2)}\to
\cX_{\lsup LM_{1,n}}^{(v_2;z_2)}
$
are isomorphic
to
$
[H^1(\Gal_F, V_{1,n})/H^0(\Gal_F, V_{1,n})]
$
and have dimension
$[F:\bQ_p]\dim V_{1,n}+v_2$.
Thus by Theorem
\ref{thm:universal}
\begin{align*}
\dim \cX_{\lsup LP_{1,n}}^{(v_2;0)}
&\le u_n + v_2 - \min_{\vec h=(*, v_2,\dots)}\delta^{\vec h}_{1,n}
\\
&\le u_n + v_2 - v_2^2/2-\min_x \delta_x.
\end{align*}
which is strictly smaller than $u_n$
unless $v_2=0,1,2$.
By local Tate duality, $\delta_x>0$ if $v_2=2$.
So we have

\begin{lem}
If $G$ is orthogonal,
then the scheme-theoretic image
of $\bigcup_{v_2=0,1}\cX_{\lsup LP_{1,n}}^{(v_2;0)}
\to \cX_{\lsup LG_n}$
contains all top-dimensional components of
$\cX_{\lsup LG_n}$.
\end{lem}

\begin{thm}
\label{thm:irr-orth}
If $G$ is orthogonal, then

(1)
$\cX_{\lsup LP_{1,n}}^{(0;0)}
\to \cX_{\lsup LM_{1,n}}
$
induces bijection of top-dimensional components.

(2)
$\cX_{\lsup LP_{1,n}}^{(1;0)}
\to \cX_{\lsup LM_{1,n}}
\to \cX_{\lsup LG_{n-2}}
$
induces bijection of top-dimensional components.
\end{thm}

\begin{proof}
Both $\cX_{\lsup LP_{1,n}}^{(0;0)}
\to \cX_{\lsup LM_{1,n}}
$
and $\cX_{\lsup LP_{1,n}}^{(1;0)}
\to \cX_{\lsup LM_{1,n}}
\to \cX_{\lsup LG_{n-2}}
$
are surjective with all fibers isomorphic to
$
[H^1(\Gal_F, V_{1,n})/H^0(\Gal_F, V_{1,n})]$.
So the theorem follows from the lemma below.
\end{proof}

\begin{lem}
If $h:\cX\to \cY$ is surjective with irreducible fibers of
constant dimension $d$
and $\dim \cX=\dim \cY+d$,
then there is a bijection between the 
corresponding top-dimensional components.
\end{lem}

\begin{proof}
Let $C\subset \cX$ be a top-dimensional irreducible component.
The scheme-theoretic image $D$ of $C$ in $\cY$
is clearly irreducible;
and is also top-dimensional by Lemma \ref{lem:fiber}.

Conversely, assume $D$ itself is a top-dimensional
irreducible component
and we want to show $h^{-1}(D)$ contains
a unique top-dimensional component $C$.
Set $D^\circ := \{x\in D| h^{-1}(x)\cap C=
h^{-1}(x)\}$.
By Lemma \ref{lem:fiber},
$h^{-1}(D-D^\circ)\cap C$
has dimension strictly smaller than 
$\dim \cX=\dim C$.
So $h^{-1}(D^\circ) \cap C=h^{-1}(D^\circ)$
must be dense in $C$
and therefore $\dim D^\circ = \dim C - d = \dim D$.
Since $D$ is an irreducible component,
$D^\circ$ is dense in $D$
and thus $\dim (D-D^\circ)<\dim D$,
which implies $\dim h^{-1}(D-D^\circ)<\dim \cX$.
\end{proof}

\subsection{The case of symplectic and unitary groups}

There are two choices for $z_2$: either $z_2=0$ or $z_2=1$.
If $z_2=0$,
then 
all fibers of
$
\cX_{\lsup LP_{1,n}}^{(v_2;0)}\to
\cX_{\lsup LM_{1,n}}^{(v_2;0)}
$
are isomorphic to
$$
[H^1(\Gal_F, \Lie U_{1,n})/H^0(\Gal_F, \Lie U_{1,n})]
$$
and have dimension equal to
exactly $[F:\bQ_p]\dim U_{1,n} + v_2$.

\begin{lem}
\label{lem:irr-unitary}
All fibers of
$
\cX_{\lsup LP_{1,n}}^{(v_2;1)}\to
\cX_{\lsup LM_{1,n}}^{(v_2;1)}
$
are isomorphic to
a hypersurface in
$$
[H^1(\Gal_F, \Lie U_{1,n})/H^0(\Gal_F, \Lie U_{1,n})]
$$
and also have dimension equal to
exactly $[F:\bQ_p]\dim U_{1,2} + v_2$.

If 
$(n-2)[F:\bQ_p]+v_2>2$,
then
all fibers of
$
\cX_{\lsup LP_{1,n}}^{(v_2;1)}\to
\cX_{\lsup LM_{1,n}}^{(v_2;1)}
$
are irreducible.
\end{lem}

\begin{proof}
It suffices to show the coarse moduli space
$|F|$
of each fiber $F$ is representable by 
an irreducible hypersurface in 
$H^1(\Gal_F, \Lie U_{1,n})$.
We can write $F\cong \bar F \times H^1(\Gal_F, Z_{1,n})$
(non-canonically)
where $\bar F \subset H^1(\Gal_F, V_{1,n})$
is the vanishing locus
of a non-degenerate symmetric bilinear form.
By Lemma \ref{lem:mixed-term},
any non-degenerate symmetric bilinear form
in more than $2$ variables
defines an irreducible hypersurface.
So, the only remaining corner case
is where $\dim H^1(\Gal_F, V_{1,n})\le 2$;
note that $\dim H^1(\Gal_F, V_{1,n}) \ge (n-2)[F:\bQ_p]+v_2$
(since $n-2=\dim V_{1,n}$ and $v_2=\dim H^2(\Gal_F, V_{1,n})$).
\end{proof}

We give an exhaustive list of
the remaining corner cases:

\begin{enumerate}
\item $[F:\bQ_p]\le 2$, $z_2=1$, $v_2=0$, $n=3$, unitary,
\item $F=\bQ_p$, $z_2=1$, $v_2=1$, $n=3$, unitary,
\item $F=\bQ_p$, $z_2=1$, $v_2=0$, $n=4$, symplectic, and
\item $F=\bQ_p$, $z_2=1$, $v_2=0$, $n=4$, unitary.
\end{enumerate}

\begin{lem}
\label{lem:corner}
In each of the cases (1)-(4) listed
above,
$\dim \cX_{\lsup LP_{1,n}}^{(v_2; z_2)}<u_n$.
\end{lem}

\begin{proof}
First consider $n=3$ and $G$ unitary.
$\cX_{\lsup LP_{1,n}}^{(v_2; z_2)}$
is 
the moduli stack of parameters of the form
$$
\begin{bmatrix}
\alpha & * & * \\
& \beta & * \\
& & \alpha(-1)
\end{bmatrix} \rtimes *
$$
The key here is that $\beta$ has only finitely many choices.
So 
the image of $\cX_{\lsup LP_{1,n}}^{(v_2; z_2)}$
in
$\cX_{\lsup LM_{1,n}}$
has dimension at most
$1- \Aut_{\Gm\times \Gm \times \Gm}(\alpha, \beta, \alpha(-1))
=-2$;
thus
$\cX_{\lsup LP_{1,n}}^{(v_2; z_2)}\le -2 + 3[F:\bQ_p]+v_2
< 3[F:\bQ_p]=u_3$.
The $n=4$ and $G$ unitary case is similar.

Next consider $n=4$ and $G$ symplectic.
$\cX_{\lsup LP_{1,n}}^{(v_2; z_2)}$
is 
the moduli stack of parameters of the form
$$
\begin{bmatrix}
\alpha & * & * \\
& \tau & * \\
& & \alpha(-1)
\end{bmatrix}
$$
The key here is that $\alpha$
is determined by $\tau$
up to finitely many choices.
So 
the image of $\cX_{\lsup LP_{1,n}}^{(v_2; z_2)}$
in
$\cX_{\lsup LM_{1,n}}$
has dimension at most
$u_2-1$;
thus
$\cX_{\lsup LP_{1,n}}^{(v_2; z_2)}\le -1 + 3[F:\bQ_p]+u_2
< u_4$.
\end{proof}

Similar to the orthogonal case,
we have
\begin{align*}
\dim \cX_{\lsup LP_{1,n}}^{(v_2;z_2)}
&\le u_n + v_2 - \min_{\vec h=(*, v_2,\dots)}\delta^{\vec h}_{1,n}
\\
&\le u_n + v_2 - v_2^2/2.
\end{align*}
which is strictly smaller than $u_n$
unless $v_2=0,1,2$.
The $v_2=2$ case is excluded by part (3)
of Theorem \ref{thm:universal}.

So we have the following

\begin{thm}
    \label{thm:decomposition}
(1)
$\bigcup_{z_2=0,1}\cX_{\lsup LP_{1,n}}^{(0;z_2)}
\to \cX_{\lsup LM_{1,n}}
$
induces a bijection of top-dimensional components.

(2)
$\bigcup_{z_2=0,1}\cX_{\lsup LP_{1,n}}^{(1;z_2)}
\to \cX_{\lsup LM_{1,n}}
\to \cX_{\lsup LG_{n-2}}
$
induces a bijection of top-dimensional components.
\end{thm}

\begin{proof}
The same proof as Theorem \ref{thm:irr-orth}
but use Lemma 
\ref{lem:irr-unitary}
and Lemma \ref{lem:corner}.
\end{proof}

We record the following technical lemma:

\begin{lem}
\label{lem:properV}
Suppose $V'\subset V_{1,n}$
is a proper $\lsup LM_{1,n}$-stable
sub-vector space scheme.
Then
$$\dim \cX_{Z_{1,n} \rtimes V' \rtimes \lsup LM_{1,n}}
<\dim \cX_{\lsup LP_{1,n}}.$$
\end{lem}

\begin{proof}
Consider fibers $F_1$ and $F_2$
of
$\cX_{Z_{1,n} \rtimes V' \rtimes \lsup LM_{1,n}}$
and
$\cX_{\lsup LP_{1,n}}$
over
some point $\spec \bFp \to \cX_{\lsup LM_{1,n}}$.
$F_2$ is either isomorphic
to $[H^1(\Gal_F, \Lie U)/H^0(\Gal_F, \Lie U)]$
or a hypersurface thereof;
$F_1$ is either isomorphic
to $[H^1(\Gal_F, Z_{1,n}\times V')/H^0(\Gal_F, Z_{1,n}\times V')]$
or a hypersurface thereof;
the condition when $F_1$ is a hypersurface
is the same as the condition when
$F_2$ is a hypersurface.
\end{proof}

\begin{cor}
(1)
The closed substack
of $\cX_{\lsup LP_{1,n}}$
consisting of parameters $\rho$
admitting
$g\in \bG_n$ and $g\not\in \lsup LP_{1,n}$
such that $g\rho g^{-1}:\Gal_F\to \lsup LG_n$
also factors through $\lsup LP_{1,n}$
has dimension strictly smaller
than $\dim \cX_{\lsup LP_{1,n}}$.

(2)
The natural map
$f:\cX_{\lsup LP_{1,n}}\to\cX_{\lsup LG_n}$
induces a bijection
of top-dimensional irreducible components.
\label{cor:mns}
\end{cor}

\begin{proof}
(1)
Write $\lsup LP_{1,n}\cap g \lsup LP_{1,n}g^{-1}
=(Z_{1,n}\rtimes V')\rtimes \lsup LM_{1,n}$
and apply Lemma \ref{lem:properV}.
By the Bruhat decomposition
$\bG_n = \brB_n W \brB_n$,
we only need to take $g$
to be representatives of the Weyl group $W$.

(2)
By Lemma \ref{lem:image-P1},
it suffices to show two distinct top-dimensional
irreducible components
$C_1\neq C_2$
 of $\cX_{\lsup LP_{1,n}}$
map to distinct top-dimensional components
$D_1\neq D_2$
 of
$\cX_{\lsup LG_n}$.
We remove from $C_1$, $C_2$, $D_1$, $D_2$
the nowhere dense closed substack
consisting of parameters
$\rho$ satisfying the condition
defined in part (1).

We claim that $x:\spec \bFp \to \cX_{\lsup LP_{1,n}}$
factors through $C_1$
if and only if $f_*x:\spec \bFp \to \cX_{\lsup LG_{n}}$
factors through $D_1$.
If $f_*x:\spec \bFp \to \cX_{\lsup LG_{n}}$
factors through $D_1$,
then there exists
$y:\spec \bFp\to C_1$
such that $f_*x=f_*y$.
So $x$ and $y$ are 
equal
up to $\bG_n$-conjugate.
Say $x$ and $y$
correspond to parameters $\rho_x$
and $\rho_y$;
we have $\rho_x = g\rho_y g^{-1}$.
We must have $g\in \bP_{1,n}$
because the bad locus has been removed.
Therefore $x=y$
and $x$ factors through $C_1$.
\end{proof}

\section{The topological Breuil-M\'ezard conjecture}

Denote by $\Chtop(\cX)$
the abelian group of top-dimensional cycles on $\cX$.

\begin{prop}
\label{prop:decomposition}
There exist natural isomorphisms
$$
\Chtop(\cX_{\lsup LM_{1,n}})
\oplus
\Chtop(\cX_{\lsup LG_{n-2}})
\xleftarrow[(p_0,p_1)]{\cong}
\Chtop(\cX_{\lsup LP_{1,n}})
\xrightarrow{\cong}
\Chtop(\cX_{\lsup LG_{n}})
$$
\end{prop}

\begin{proof}
Combine Theorem \ref{thm:decomposition}
and Corollary 
\ref{cor:mns}.
\end{proof}

Denote by $\alpha_1^\vee:\Gm\to \bT_n$
the {\it unique} simple positive coroot
(up to $\Gal(K/F)$-permutation)
of $\bG_n$
appearing in $\Lie U_{1,n}$,
which induces
a map
$$
\lsup L\Res_{K/F}\Gm \to \lsup LT_n;
$$
so $\alpha_1^\vee$ induces a natural isomorphism
$$
\cX_{\lsup LM_{1,n}}
\cong
\cX_{\lsup L\Res_{K/F}\Gm}
\times
\cX_{\lsup LG_{n-2}}.
$$
For the sake of uniformity,
we denote by $\{\infty\}\times \Chtop(\cX_{\lsup LG_{n-2}})$
the image of $p_1$ in Proposition \ref{prop:decomposition},
and rewrite $(p_0, p_1)$
as
$$
\Chtop(\cX_{\lsup LP_{1,n}})
\xrightarrow{\cong}
(\Chtop(\cX_{\lsup L\Res_{K/F}\Gm}) \cup \{\infty\})
\times
\Chtop(\cX_{\lsup LG_{n-2}}).
$$

\begin{dfn}
Use the notation
$$
\text{($\dagger$)}\hspace{10mm}
X_1(K, \Gm)
:=
\Chtop(\cX_{\lsup L\Res_{K/F}\Gm}) \cup \{\infty\}
$$
and denote by
$$
\alpha_1^\vee:
\Chtop(\cX_{\lsup LG_n})
\to 
X_1(K, \Gm)
$$
the composite map.
\label{dfn:x1}
\end{dfn}

\begin{thm}
    \label{thm:main2}
Denote by $\{\alpha_1, \dots, \alpha_r\}$
the simple positive roots
so that the Dynkin diagram of $\bG_n$
is
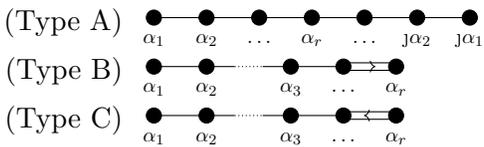
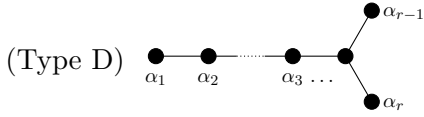
\begin{itemize}
\item[(Type A)] 
\dynkin[
  labels={\alpha_1,\alpha_2,\dots,
        \alpha_r, \dots
        , \j\alpha_2,\j\alpha_1},
  scale=2
]{A}{7}
\item[(Type B)] 
\dynkin[
  labels={\alpha_1,\alpha_2,\alpha_3,\dots,\alpha_r},
  scale=2
]{B}{}
\item[(Type C)] 
\dynkin[
  labels={\alpha_1,\alpha_2,\alpha_3,\dots,\alpha_r},
  scale=2
]{C}{}
\item[(Type D)] 
\dynkin[
labels={\alpha_1,\alpha_2,\alpha_3,\dots,\alpha_{r-1}, \alpha_{r}},
  scale=2
]{D}{}
\end{itemize}
There exist
natural isomorphisms
\begin{itemize}
\item[($A_{2n}$)]
$
(\alpha_1^\vee, \dots,\alpha_n^\vee, \Remain
):
\Chtop(\cX_{\lsup LU_{2n+1}})
\xrightarrow{\cong} X_1(K, G_m)^{\times n}
\times 
\Chtop(\cX_{\lsup LU_{1}})
$,
\item[($A_{2n+1}$)]
$
(\alpha_1^\vee, \dots,\alpha_n^\vee, \Remain
):
\Chtop(\cX_{\lsup LU_{2n+2}})
\xrightarrow{\cong} X_1(K, G_m)^{\times n}
\times 
\Chtop(\cX_{\lsup LU_{2}})
$,
\item[($B_{n}$)]
$
(\alpha_1^\vee, \dots,\alpha_{n-1}^\vee, \Remain
):
\Chtop(\cX_{\GSp_{2n}})
\xrightarrow{\cong} X_1(K, G_m)^{\times (n-1)}
\times 
\Chtop(\cX_{\GL_2})
$,
\item[($C_{n}$)]
$
(\alpha_1^\vee, \dots,\alpha_{n-1}^\vee, \Remain
):
\Chtop(\cX_{\GSO_{2n+1}})
\xrightarrow{\cong}X_1(K, G_m)^{\times (n-1)}
\times 
\Chtop(\cX_{\GSO_3})
$,
\item[($D_{n}$)]
$
(\alpha_1^\vee, \dots,\alpha_{n-1}^\vee, \Remain
):
\Chtop(\cX_{\GSO_{2n}})
\xrightarrow{\cong} X_1(K, G_m)^{\times (n-2)}
\times 
\Chtop(\cX_{\GSO_4})
$,
\end{itemize}
\end{thm}

\begin{cor}
\label{cor:main3}
All top-dimensional irreducible components
are covered by algebraic cycles
defined by reduction mod $p$
of de Rham stacks of regular Hodge type.
\end{cor}

\begin{proof}
It follows immediately from
our companion paper \cite{L25A}
which constructs a machinery that
turn results in this paper
to the unconditional existence of de Rham lifts
of regular Hodge type.
\end{proof}

\begin{rmk}
For readers who don't want to appeal to \cite{L25A},
we can construct the de Rham stacks explicitly.
All de Rham stacks we use
will be of ``smallest possible inertial type'' $\tau_0$,
where $\tau_0:I_F \to \lsup LG$,
$\gamma\mapsto 1\rtimes \gamma$.
For ease of notation, we set $F=\bQ_p$
but the same construction works for arbitrary fields.
We construct the de Rham stacks inductively on $n$:
suppose
for each tuple
$$
\vec a=(\alpha_2^\vee(C), \dots, \alpha_{n-1}^\vee(C), \Remain(C))
$$
(where $C$ is a top-dimensional irreducible
component of
$\lsup LG_n$),
we have associated to it
a Hodge type $\lambda'$
so that the de Rham stack
$\cX_{\lsup LG_{n-2}}^{\lambda', \tau_0}$
covers the component $\vec a$.
If $\alpha_1^\vee(C)\neq \infty$,
we let $\lambda$ be the unique
cocharacter extending $\lambda'$
such that
$
\alpha_1(\lambda) = 
2p -\alpha_1^\vee(C)
$,
anf $\alpha_1(\lambda) = 1$ if otherwise
(it depends on the positivity convention
of Hodge types,
use $\alpha_1(\lambda) = 
\alpha_1^\vee(C)-2p-1$ or $\alpha_1^\vee(\lambda)=-1$
if the negative convention is used.)
\end{rmk}

\addcontentsline{toc}{section}{References}
\printbibliography

@book {Ko02,
    AUTHOR = {Koch, Helmut},
     TITLE = {Galois theory of {$p$}-extensions},
    SERIES = {Springer Monographs in Mathematics},
      NOTE = {With a foreword by I. R. Shafarevich,
              Translated from the 1970 German original by Franz Lemmermeyer,
              With a postscript by the author and Lemmermeyer},
 PUBLISHER = {Springer-Verlag, Berlin},
      YEAR = {2002},
     PAGES = {xiv+190},
      ISBN = {3-540-43629-4},
   MRCLASS = {11S25 (11R32 11R34 11S20)},
  MRNUMBER = {1930372},
       DOI = {10.1007/978-3-662-04967-9},
       URL = {https://doi.org/10.1007/978-3-662-04967-9},
}

@misc{L23,
      title={A Deligne-Lusztig type correspondence for tame $p$-adic groups},
      author={Zhongyipan Lin},
      year={2023},
      eprint={2306.02093},
      archivePrefix={arXiv},
      primaryClass={math.NT}
}

@misc {L23B,
     title={The Emerton-Gee stacks for tame groups, I},
      author={Zhongyipan Lin},
      year={2023},
      eprint={2304.05317},
      archivePrefix={arXiv},
      primaryClass={math.NT}
}

@article{BG19,
  title={G-valued local deformation rings and global lifts},
  author={Rebecca Bellovin and T. Gee},
  journal={Algebra and Number Theory},
  volume={13},
  pages={333-378},
  year={2019}
}

@book{EG23,
        title = {Moduli stacks of \'etale $(\varphi,\Gamma)$-modules and the existence of crystalline lifts},
        author = {Matthew Emerton and Toby Gee},
        book = {Annals of Math. Studies},
        year = {2023}
    }

@book{Stacks,
        title = {The Stacks project 2025},
        author = {Stacks project authors},
    }

@misc {L25A,
    author = {Lin, Zhongyipan},
    title = {Heisenberg varieties and the existence of de Rham lifts},
      year={2025},
      archivePrefix={Forthcoming},
}

@misc {L25B,
    author = {Lin, Zhongyipan},
    title = {The Emerton-Gee stacks for tame groups, II},
      year={2025},
      archivePrefix={Forthcoming},
}
\end{document}